\documentclass{siamart}

\usepackage{enumitem}

\usepackage{lipsum}
\usepackage{amsfonts}
\usepackage{graphicx}
\usepackage{epstopdf}
\usepackage{algorithmic}
\ifpdf
  \DeclareGraphicsExtensions{.eps,.pdf,.png,.jpg}
\else
  \DeclareGraphicsExtensions{.eps}
\fi

\title{A Primal-Dual Parallel Method with $O(1/\epsilon)$ Convergence for Constrained Composite Convex Programs\thanks{This paper extends our conference paper \cite{YuNeely16CDC} by considering composite convex programs and proposing new algorithm parameter rules that are irrelevant to the optimal Lagrange multipliers.}}

\author{
  Hao Yu\thanks{Department of Electrical Engineering, University of Southern California, Los Angeles, CA (\email{yuhao@usc.edu},
    \email{mjneely@usc.edu}).}
  \and
  Michael J. Neely\footnotemark[2]
}

\newtheorem{Def}{Definition}

\newtheorem{Assumption}{Assumption}

\newtheorem{Thm}{Theorem}

\newtheorem{Lem}{Lemma}
\newtheorem{Cor}{Corollary}

\DeclareMathOperator*{\argmin}{argmin}
\newcommand*{\tran}{^{\mkern-1.5mu\mathsf{T}}}

\begin{document}
\maketitle

\begin{abstract}
This paper considers large scale constrained convex (possibly composite and non-separable) programs, which are usually difficult to solve by interior point methods or other Newton-type methods due to the non-smoothness or the prohibitive computation and storage complexity for Hessians and matrix inversions.  Instead, they are often solved by first order gradient based methods or decomposition based methods.  The conventional primal-dual subgradient method, also known as the Arrow-Hurwicz-Uzawa subgradient method, is a low complexity algorithm with an $O(1/\epsilon^2)$ convergence time. Recently, a new Lagrangian dual type algorithm with a faster $O(1/\epsilon)$ convergence time is proposed in Yu and Neely (2017).  However, if the objective or constraint functions are not separable, each iteration of the Lagrangian dual type method in Yu and Neely (2017) requires to solve a unconstrained convex program, which can have huge complexity. This paper proposes a new primal-dual type algorithm with $O(1/\epsilon)$ convergence for general constrained convex programs. Each iteration of the new algorithm can be implemented in parallel with low complexity even when the original problem is composite and non-separable. 

\end{abstract}

\begin{keywords}
constrained convex programs, composite convex programs, parallel methods, convergence time 
\end{keywords}

\begin{AMS}
90C25, 90C30
\end{AMS}

\section{Introduction}\label{sec:intro}
Recall that a function $h(\mathbf{x})$ is said to be separable (with respect to its vector variable $\mathbf{x}$) if it can be written as the summation of multiple smaller functions, each of which only involves disjoint components or blocks of $\mathbf{x}$, e.g, $h(\mathbf{x}) = \sum_{i=1}^{n} h^{(i)}(x_{i})$. Fix positive integers $n$ and $m$, which are typically large.  Consider the following constrained convex program:
\begin{align}
\text{min} \quad &F(\mathbf{x}) \overset{\Delta}{=} f(\mathbf{x}) + \tilde{f}(\mathbf{x})\label{eq:program-objective}\\
\text{s.t.} \quad  &  G_k(\mathbf{x}) \overset{\Delta}{=} g_k(\mathbf{x})  + \tilde{g}_k(\mathbf{x})\leq  0, \forall k\in\{1,2,\ldots,m\} \label{eq:program-inequality-constraint}\\
			 &  \mathbf{x}\in \mathcal{X} \label{eq:program-set-constraint}
\end{align}
where set $\mathcal{X}\subseteq \mathbb{R}^{n}$ is a closed convex set;  function $f(\mathbf{x})$ is convex and smooth (but possibly non-separable) on $\mathcal{X}$; function $\tilde{f}(\mathbf{x})$ is convex and  separable  (but possibly non-smooth) on $\mathcal{X}$; functions $g_k(\mathbf{x}),\forall k \in\{1,2,\ldots,m\}$ are convex,  Lipschitz continuous and smooth (but possibly non-separable) on $\mathcal{X}$, and functions $\tilde{g}_k(\mathbf{x})$ are convex, Lipschitz continuous and separable (but possibly non-smooth) on $\mathcal{X}$.  The convex program \eqref{eq:program-objective}-\eqref{eq:program-set-constraint} is called a constrained composite convex program since either its objective function $F(\mathbf{x})$ or each of its constraint functions $G_k(\mathbf{x})$ is in general the sum of a smooth function and a non-smooth function.

Denote the stacked vector of functions via $\mathbf{g}(\mathbf{x}) = [g_1(\mathbf{x}), g_2(\mathbf{x}), \ldots, g_m(\mathbf{x})]\tran$; $\tilde{\mathbf{g}}(\mathbf{x}) = \big[\tilde{g}_1(\mathbf{x}), \tilde{g}_2(\mathbf{x}), \ldots, \tilde{g}_m(\mathbf{x})\big]\tran$ and $\mathbf{G}(\mathbf{x}) = \big[G_1(\mathbf{x}), G_2(\mathbf{x}), \ldots, G_m(\mathbf{x})\big]\tran$.  The Lipschitz continuity of each $g_{k}(\mathbf{x})$ and $\tilde{g}_k(\mathbf{x})$ implies that $\mathbf{g}(\mathbf{x}) + \tilde{\mathbf{g}}(\mathbf{x})$ is Lipschitz continuous on $\mathcal{X}$. Throughout this paper, we use $\Vert\mathbf{x}\Vert$ to denote the Euclidean norm of vector $\mathbf{x}$, also known as the $l_2$ norm, and have the following assumptions on convex program  \eqref{eq:program-objective}-\eqref{eq:program-set-constraint}:

\begin{Assumption}[Basic Assumptions] \label{as:basic}~
\begin{itemize}
\item  There exists a (possibly non-unique) optimal solution $\mathbf{x}^\ast\in \mathcal{X}$ that solves convex program \eqref{eq:program-objective}-\eqref{eq:program-set-constraint}. 
\item There exists $\beta>0$ such that $\Vert \mathbf{G}(\mathbf{x}) - \mathbf{G}(\mathbf{y}) \Vert \leq \beta \Vert \mathbf{x} - \mathbf{y}\Vert$ for all $\mathbf{x}, \mathbf{y} \in \mathcal{X}$, i.e., $\mathbf{G}(\mathbf{x})$ is Lipschitz continuous with modulus $\beta$. 
\end{itemize}
\end{Assumption}

\begin{Assumption}[Existence of Lagrange multipliers]\label{as:strong-duality}~ 
There exists a Lagrange multiplier vector $\boldsymbol{\lambda}^\ast = [\lambda_1^\ast, \lambda_2^\ast, \ldots, \lambda_m^\ast]\geq \mathbf{0}$ attaining the strong duality for problem \eqref{eq:program-objective}-\eqref{eq:program-set-constraint}, i.e., 
\begin{align*}
q(\boldsymbol{\lambda}^\ast) = \min\limits_{\mathbf{x}\in \mathcal{X}}\left\{F(\mathbf{x}) : G_k(\mathbf{x}) \leq 0, \forall k\in\{1,2,\ldots,m\}\right\}, 
\end{align*}
where $q(\boldsymbol{\lambda}) = \min\limits_{\mathbf{x}\in \mathcal{X}}\{F(\mathbf{x})+ \sum_{k=1}^m \lambda_k G_k(\mathbf{x})\}
$ is the {\it Lagrangian dual function} of problem \eqref{eq:program-objective}-\eqref{eq:program-set-constraint}.
\end{Assumption}


Under \cref{as:basic} and \cref{as:strong-duality},  this paper proposes a new primal dual type algorithm which can solve convex program \eqref{eq:program-objective}-\eqref{eq:program-set-constraint} with $O(1/\epsilon)$ convergence. That is, the new algorithm only requires $O(1/\epsilon)$ iterations to achieve an $\epsilon$-approximate solution.  Furthermore, each iteration of this new algorithm can be decomposed into multiple smaller independent subproblems and hence can be implemented in parallel with low complexity even though the original convex program \eqref{eq:program-objective}-\eqref{eq:program-set-constraint} involves non-separable $f(\mathbf{x})$ and $\mathbf{g}(\mathbf{x})$.

\subsection{Example Problems}  
The general convex program \eqref{eq:program-objective}-\eqref{eq:program-set-constraint} considered in this paper includes many difficult convex programs as special cases.

\subsubsection{Large Scale Constrained Smooth Convex Programs}

If $\tilde{f}(\mathbf{x}) \equiv  0$ and $\tilde{\mathbf{g}}(\mathbf{x}) \equiv  \mathbf{0}$, then problem \eqref{eq:program-objective}-\eqref{eq:program-set-constraint} is a constrained smooth convex program.  In general,  constrained smooth convex program \eqref{eq:program-objective}-\eqref{eq:program-set-constraint} can be solved via interior point methods  (or other Newton type methods) which involve the computation of Hessians and matrix inversions at each iteration.  The associated computation complexity and memory space complexity at each iteration is between $O(n^{2})$ and $O(n^{3}$), which is prohibitive when $n$ is extremely large. For example, if $n=10^{5}$ and each floating point number uses $4$ bytes, then $40$ Gbytes of memory space is required even to save the Hessian at each iteration.  Thus, large scale convex programs are usually solved by first order gradient based methods or decomposition based methods. The primal-dual subgradient algorithm, also known as the Arrow-Hurwicz-Uzawa subgradient method, is a first order method with a slow $O(1/\epsilon^2)$ convergence time for large scale convex programs \cite{Nedic09_PrimalDualSubgradient}.

\subsubsection{Constrained Composite Convex Programs}
If $\tilde{f}(\mathbf{x})\not\equiv 0$ and/or $\tilde{\mathbf{g}}(\mathbf{x}) \not\equiv \mathbf{0}$, then problem \eqref{eq:program-objective}-\eqref{eq:program-set-constraint} is a constrained composite convex program. Due to the non-differentiability, interior points methods (or other Netwon type methods) are usually not applicable.   Such a non-smooth convex program can be solved by a mirror descent based method in \cite{Beck10ORL} with a slow $O(1/\epsilon^2)$ convergence time.  In the special case when there is only one single smooth constraint given by $g_1(\mathbf{x}) \leq 0$, i.e., $\tilde{g}_1(\mathbf{x})\equiv 0$, work \cite{Ron16MP} proposes a dual method with an $O(1/\epsilon)$ convergence time.  In the special case when $\mathbf{g}(\mathbf{x}) = \mathbf{A}\mathbf{x} - \mathbf{b}$ is linear and $\tilde{\mathbf{g}}(\mathbf{x}) \equiv \mathbf{0}$,  work \cite{ShuzhongZhang16RandomPrimalDual}  proposes a random primal-dual method that can converges to a solution whose expected error is $\epsilon$ with an $O(1/\epsilon)$ convergence time.

One representative example of constrained composite convex programs is the constrained LASSO problem from machine learning applications \cite{James12ConstrainedLasso} and financial portfolio optimization \cite{Brodie09PNAS} as follows:
\begin{align*}
\text{min}_{\mathbf{x}} \quad &\Vert \mathbf{A}\mathbf{x} - \mathbf{b}\Vert^2 + \lambda \Vert\mathbf{x}\Vert_1\\
\text{s.t.~} \quad  &  \mathbf{C}\mathbf{x} -\mathbf{d} \leq \mathbf{0}\\
& \mathbf{x}\in \mathcal{X}
\end{align*}
where $\Vert \mathbf{x}\Vert_1 = \sum_{i=1}^n |x_i|$ denotes the $l_{1}$ norm of vector $\mathbf{x}$. The constrained LASSO problem is a special case of convex program \eqref{eq:program-objective}-\eqref{eq:program-set-constraint} with $f(\mathbf{x})=\Vert \mathbf{A}\mathbf{x} - \mathbf{b}\Vert^2$, $\tilde{f}(\mathbf{x}) = \lambda \Vert \mathbf{x}\Vert_1$, $\mathbf{g}(\mathbf{x}) = \mathbf{C}\mathbf{x} - \mathbf{d}$ and $\tilde{\mathbf{g}}(\mathbf{x}) \equiv \mathbf{0}$.  In fact, many constrained optimization problems from machine learning, compressed sensing and financial portfolio optimization involve a non-smooth but separable $l_1$ norm $\Vert \mathbf{x}\Vert_1$ term in the objective or constraint functions and can hence can be cast as a special case of convex program \eqref{eq:program-objective}-\eqref{eq:program-set-constraint}.

\subsection{The Primal-Dual Subgradient Method}
The primal-dual subgradient method, also known as Arrow-Hurwicz-Uzawa subgradient method, with primal averaging is a first order method that can be applied to solve convex program \eqref{eq:program-objective}-\eqref{eq:program-set-constraint} as described in \cref{alg:primal-dual-subgradient}.  In this paper, we use $\nabla h (\mathbf{x}(t-1))$ to denote either the gradient (when $h(\cdot)$ is differentiable) or a subgradient (when $h(\cdot)$ is non-differentiable) of function $h(\mathbf{x})$ at point $\mathbf{x} = \mathbf{x}(t-1)$. 

The primal-dual subgradient method can solve constrained non-smooth convex programs. The updates of $\mathbf{x}(t)$ and $\boldsymbol{\lambda}(t)$ only involve the computation of subgradients and simple projection operations.  For large scale constrained smooth convex programs with large $n$ value, the computation of subgradients is much simpler than the computation of Hessians and matrix inversions and hence the primal dual subgradient has lower complexity computations at each iteration and is more suitable when compared with the interior point method.  However, \cref{alg:primal-dual-subgradient} is known to have a slow $O(1/\epsilon^2)$ convergence time \cite{Nedic09_PrimalDualSubgradient}.  Another drawback of \cref{alg:primal-dual-subgradient} is that its implementation requires $\lambda_k^{\max}$, which are upper bounds of each component of the Lagrange multiplier vector $\boldsymbol{\lambda}^{\ast}$ that attains the strong duality.  In practice, $\boldsymbol{\lambda}^{\ast}$ is usually unavailable.

\begin{algorithm} 
\caption{The Primal-Dual Subgradient Algorithm}
\label{alg:primal-dual-subgradient}
Let $c>0$ be a constant step size. Choose any $\mathbf{x}(0) \in \mathcal{X}$. Initialize Lagrangian multipliers $\lambda_{k}(0) = 0, \forall k\in\{1,2,\ldots, m\}$.  At each iteration $t\in\{1,2,\ldots\}$, observe $\mathbf{x}(t-1)$ and $\boldsymbol{\lambda}(t-1)$ and do the following:
\begin{itemize}
\item  Choose $\mathbf{x}(t)$ via 
\begin{align*}
\mathbf{x}(t)  =\mathcal{P}_{\mathcal{X}} \Big[ \mathbf{x}(t-1) - c\big[ \nabla F(\mathbf{x}(t-1)) +\sum_{k=1}^{m} \lambda_{k}(t-1) \nabla G_{k}(\mathbf{x}(t-1))\big]\Big] ,
\end{align*}
where $\mathcal{P}_{\mathcal{X}}[\cdot]$ is the projection onto convex set $\mathcal{X}$.
\item Update Lagrangian multipliers $\boldsymbol{\lambda}(t)$ via 
\begin{align*}
\lambda_{k}(t) = \left[ \lambda_{k}(t-1) + c G_{k}(\mathbf{x}(t-1))\right]_{0}^{\lambda_{k}^{\max}}, \forall k\in\{1,2,\ldots, m\},
\end{align*}
where $\lambda_{k}^{\max} > \lambda_{k}^{\ast}$ and $[\cdot ]_{0}^{\max}$ is the projection onto interval $[0,\lambda_{k}^{\max}]$.
\item Update the running averages $\overline{\mathbf{x}}(t)$ via
\begin{align*}
\overline{\mathbf{x}}(t+1) = \frac{1}{t+1} \sum_{\tau=0}^{t} \mathbf{x}(\tau) = \overline{\mathbf{x}}(t) \frac{t}{t+1} + \mathbf{x}(t) \frac{1}{t+1}
\end{align*}
\end{itemize}
\end{algorithm}

\subsection{The Dual Subgradient Method and Its Variations}
The classical dual subgradient algorithm is a Lagrangian dual type iterative method that can solve constrained strictly convex programs \cite{book_NonlinearProgramming_Bertsekas}. By averaging the resulting primal estimates from the classical dual subgradient algorithm, we can solve general constrained convex programs (possibly without strict convexity) with an $O(1/\epsilon^2)$ convergence time \cite{Neely05DCDIS,Nedic09,Neely14Arxiv_ConvergenceTime}. The dual subgradient algorithm with primal averaging is more suitable to separable convex programs because the updates of each component $x_{i}(t)$ are independent and parallel if both the objective function and the constraint function are separable $\mathbf{x}$.

Recently, a new Lagrangian dual type algorithm with $O(1/\epsilon)$ convergence for general convex programs is proposed in \cite{YuNeely17SIOPT}.  This algorithm can solve convex program \eqref{eq:program-objective}-\eqref{eq:program-set-constraint} following the steps described in \cref{alg:general-alg}. Similar to the dual subgradient algorithm with primal averaging, \cref{alg:general-alg} can decompose the updates of $\mathbf{x}(t)$ into smaller independent subproblems if functions $F(\mathbf{x})$ and $G_k(\mathbf{x})$ are separable. Moreover, \cref{alg:general-alg} has faster $O(1/\epsilon)$ convergence when compared with the primal-dual subgradient algorithm or the dual subgradient algorithm with primal averaging.

\begin{algorithm} 
\caption{Algorithm 1 in \cite{YuNeely17SIOPT}}
\label{alg:general-alg}
Let $\alpha>0$ be a constant parameter. Choose any $\mathbf{x}(-1) \in \mathcal{X}$. Initialize virtual queues $Q_{k}(0) = \max\{0, -G_{k}(\mathbf{x}(-1))\} , \forall k\in\{1,2,\ldots, m\}$. At each iteration $t\in\{0,1,2,\ldots\}$, observe $\mathbf{x}(t-1)$ and $\mathbf{Q}(t)$ and do the following:
\begin{itemize}
\item  Choose $\mathbf{x}(t)$ as 
\begin{align*}
\mathbf{x}(t)  =\argmin_{\mathbf{x}\in \mathcal{X}} \Big\{ F(\mathbf{x})  + [\mathbf{Q}(t) + \mathbf{G}(\mathbf{x}(t-1))]\tran \mathbf{G}(\mathbf{x}) +  \alpha \Vert \mathbf{x} - \mathbf{x}(t-1)\Vert^{2}\Big\}.
\end{align*}

\item Update virtual queue vector $\mathbf{Q}(t)$ via  \[Q_{k}(t+1) = \max\{-G_{k}(\mathbf{x}(t)), Q_{k}(t) + G_{k}(\mathbf{x}(t))\}, \forall k\in\{1,2,\ldots, m\}.\]
\item Update the running averages $\overline{\mathbf{x}}(t)$ via \[\overline{\mathbf{x}}(t+1) = \frac{1}{t+1}\sum_{\tau=0}^{t} \mathbf{x}(\tau) = \overline{\mathbf{x}}(t) \frac{t}{t+1} + \mathbf{x}(t) \frac{1}{t+1}.\]
\end{itemize}
\end{algorithm}

In this paper, however, objective function $F(\mathbf{x})$ involves possibly non-separable $f(\mathbf{x})$ and each constraint function $G_k(\mathbf{x})$ involves possibly non-separable $g_{k}(\mathbf{x})$. As a result, the update of $\mathbf{x}(t)$ is not decomposable and requires to solve a set constrained non-smooth strongly convex program which is typically solved via a subgradient based method. However, the subgradient based method for set constrained convex programs is an iterative technique and involves at least one projection operation at each iteration.  To obtain an $\epsilon$-approximate solution to the set constrained convex program, the projected subgradient method in general requires $O(1/\epsilon^2)$ iterations and can be improved to only require $O(1/{\epsilon})$ or $O(1/\sqrt{\epsilon})$ iterations for certain special problems \cite{book_ConvexOpt_Nesterov,Beck2009FISTA}.

\subsection{New Algorithm}
In this paper, we propose a new primal-dual type algorithm to solve convex program \eqref{eq:program-objective}-\eqref{eq:program-set-constraint} as described in \cref{alg:new-alg}. The new algorithm uses the same virtual queue update as in \cref{alg:general-alg}, however, the update of $\mathbf{x}(t)$ is fundamentally different. The modification enable us to update each component of $\mathbf{x}$ in parallel even if $f(\mathbf{x})$ or each $g_{k}(\mathbf{x})$ is non-separable. Later, we will further show that the $O(1/\epsilon)$ convergence of \cref{alg:general-alg} is preserved in the new algorithm.

\begin{algorithm} 
\caption{New Algorithm}
\label{alg:new-alg}
Let $\{\alpha(t), t\geq0\}$ be a sequence of positive algorithm parameters (defined in Section \ref{sec:convergence-time}). Choose any $\mathbf{x}(-1) \in \mathcal{X}$. Initialize virtual queues $Q_{k}(0) = \max\{0, -G_{k}(\mathbf{x}(-1))\} , \forall k\in\{1,2,\ldots, m\}$. At each iteration $t\in\{0,1,\ldots\}$, observe $\mathbf{x}(t-1)$ and $\mathbf{Q}(t)$ and do the following:
\begin{itemize}
\item
Choose $\mathbf{x}(t)$ to solve  $\min_{\mathbf{x}\in \mathcal{X}} \Big\{ [\nabla f(\mathbf{x}(t-1))]\tran \mathbf{x}  + \tilde{f}(\mathbf{x}) + \sum_{k=1}^m \big[Q_k(t) + G_k(\mathbf{x}(t-1))\big]\big[[\nabla g_k(\mathbf{x}(t-1))]\tran \mathbf{x} + \tilde{g}_k(\mathbf{x})\big]  +  \alpha(t) \Vert \mathbf{x} - \mathbf{x}(t-1)\Vert^{2}\Big\}$.

\item Update virtual queue vector $\mathbf{Q}(t)$ via  \[Q_{k}(t+1) = \max\{-G_{k}(\mathbf{x}(t)), Q_{k}(t) + G_{k}(\mathbf{x}(t))\}, \forall k\in\{1,2,\ldots, m\}.\]
\item Update the running averages $\overline{\mathbf{x}}(t)$ via \[\overline{\mathbf{x}}(t+1) = \frac{1}{t+1}\sum_{\tau=0}^{t} \mathbf{x}(\tau) = \overline{\mathbf{x}}(t) \frac{t}{t+1} + \mathbf{x}(t) \frac{1}{t+1}.\]
\end{itemize}
\end{algorithm}

At each iteration $t\in\{0,1,2,\ldots\}$, $Q_k(t) + G_k(\mathbf{x}(t-1)), \forall k\in\{1,2,\ldots,m\}$ are given constants.  Note that $\nabla f(\mathbf{x}(t-1))]\tran \mathbf{x} + \sum_{k=1}^m \big[Q_k(t) + G_k(\mathbf{x}(t-1))\big] \big[[\nabla g_k(\mathbf{x}(t-1))]\tran \mathbf{x}\big]$ is a linear function and hence is separable; $\tilde{f}(\mathbf{x}) +  \sum_{k=1}^m [Q_k(t) + G_k(\mathbf{x}(t-1))] \tilde{g}_k(\mathbf{x})$ is separable by assumption; and $\alpha(t) \Vert \mathbf{x}-\mathbf{x}(t-1)\Vert^2$ is also separable. Thus, the update of $\mathbf{x}(t)$ requires to minimize a separable convex function. It follows that each component of $\mathbf{x}(t)$ can be updated independently by solving a scalar convex program.  Thus, each iteration of \cref{alg:new-alg} is parallel and has low complexity.

The next lemma shows that if $\tilde{f}(\mathbf{x})$ and each $\tilde{g}_{k}(\mathbf{x})$ are $l_{1}$ norms, then the $\mathbf{x}(t)$ update in \cref{alg:new-alg} has a closed-form update equation for each coordinate. 

\begin{Lem}\label{lm:l1-norm-closed-form}
If $\tilde{f}(\mathbf{x})= c_{0} \Vert \mathbf{x}\Vert_{1}$, $\tilde{g}_{k}(\mathbf{x}) = c_{k} \Vert \mathbf{x}\Vert_{1}, \forall k\in\{1,2,\ldots,m\}$ and $\mathcal{X} = \prod_{i=1}^{n}[x_{i}^{\min}, x_{i}^{\max}]$, then the following holds:
\begin{enumerate}
\item The update of $\mathbf{x}(t)$ in \cref{alg:new-alg} can be decomposed into $n$ scalar convex programs and each $x_{i}(t), i\in\{1,2,\ldots,n\}$ is the solution to a scalar convex program given by
\begin{align}
\min_{x_{i}\in [x_i^{\min}, x_i^{\max}]} \Big\{ \alpha [x_{i} - x_{i}(t-1)]^{2} +d_{i} x_{i} +  e_{i}|x_{i}| \Big\}, \label{eq:lm-l1-norm-min-scalar-problem}
\end{align}
where $\alpha=\alpha(t)>0$, $d_{i} = \frac{\partial f(\mathbf{x}(t-1))}{\partial x_{i}}  + \sum_{k=1}^{m} [Q_{k}(t) + G_{k}(\mathbf{x}(t-1))] \frac{\partial g_{k}(\mathbf{x}(t-1))}{\partial x_{i}}$ and $e_{i} = c_{0} + \sum_{k=1}^{m} [Q_{k}(t) + G_{k}(\mathbf{x}(t-1))] c_{k}$ are constants. Note that we use $\frac{\partial \phi(\mathbf{x}(t-1))}{\partial x_{i}}$ to denote the partial gradient of $\phi(\mathbf{x})$ with respect to the $i$-th component $x_{i}$ at point $\mathbf{x} = \mathbf{x}(t-1)$. 
\item Scalar convex program \eqref{eq:lm-l1-norm-min-scalar-problem} has a closed-form solution given by 
\begin{align*}
x_{i}^{\ast} = \left\{\begin{array}{cl} \big[x_{i}(t-1) - \frac{d_{i}}{2\alpha}- \frac{e_{i}}{2\alpha}\big]_{x_i^{\min}}^{x_i^{\max}},&\text{if}~ x_{i}(t-1) - \frac{d_{i}}{2\alpha} > \frac{e_{i}}{2\alpha},\\   \big[x_{i}(t-1) - \frac{d_{i}}{2\alpha}+ \frac{e_{i}}{2\alpha}\big]_{x_i^{\min}}^{x_i^{\max}} ,&\text{if}~ x_{i}(t-1) - \frac{d_{i}}{2\alpha} < -\frac{e_{i}}{2\alpha}, \\ \big[0\big]_{x_i^{\min}}^{x_i^{\max}}, &\text{else}.\end{array}\right.
\end{align*} 
where $[\cdot]_a^b$ is the projection onto the interval $[a,b]$.
\end{enumerate}
\end{Lem}

\begin{proof}~
\begin{enumerate}
\item The first part follows trivially by combining terms of each component $x_{i}$ in the convex program with vector variable $\mathbf{x}$.
\item The second part follows by recalling that the subgradient of $|x|$ is $1$ when $x>0$; is $-1$ when $x<0$; and is the interval $[-1,1]$ when $x=0$. The closed-form solution is obtained by considering different ranges of parameters.
\end{enumerate}
\end{proof}

The next lemma summarizes that if $\tilde{f}(\mathbf{x})\equiv 0$ and $\tilde{\mathbf{g}}(\mathbf{x})\equiv \mathbf{0}$, i.e., problem \eqref{eq:program-objective}-\eqref{eq:program-set-constraint} is a constrained smooth convex program, then $\mathbf{x}(t)$ update in \cref{alg:new-alg} follows a simple projected gradient update, which is parallel for each component as long as $\mathcal{X}$ is a Cartesian product.

\begin{Lem} \label{lm:x-update-proj-implementation}
If $\tilde{f}(\mathbf{x})\equiv 0$ and $\tilde{\mathbf{g}}(\mathbf{x})\equiv \mathbf{0}$, then the update of $\mathbf{x}(t)$ in \cref{alg:new-alg} is given by
\begin{align*}
\mathbf{x}(t)  =\mathcal{P}_{\mathcal{X}}\Big[ \mathbf{x}(t-1) -\frac{1}{2\alpha(t)} \mathbf{d}(t) \Big]
\end{align*}
where $\mathbf{d}(t) = \nabla f(\mathbf{x}(t-1)) + \sum_{k=1}^{m} [Q_{k}(t) + g_{k}(\mathbf{x}(t-1))] \nabla g_{k}(\mathbf{x}(t-1))$ and $\mathcal{P}_{\mathcal{X}}[\cdot]$ is the projection onto convex set $\mathcal{X}$.
\end{Lem}
\begin{proof}
The projection operator can be reinterpreted as an optimization problem as follows:
\begin{align}
\mathbf{x}(t) =&  \mathcal{P}_{\mathcal{X}} \Big[ \mathbf{x}(t-1) - \frac{1}{2\alpha(t)} \mathbf{d}(t)\Big] \nonumber\\
\overset{(a)}{=}& \argmin_{\mathbf{x}\in \mathcal{X}} \Big[ \big\Vert \mathbf{x} - [\mathbf{x}(t-1)- \frac{1}{2\alpha(t)} \mathbf{d}(t)] \big\Vert^{2}\Big] \nonumber\\
=&\argmin_{\mathbf{x}\in \mathcal{X}} \Big[ \Vert \mathbf{x} -\mathbf{x}(t-1) \Vert^{2} + \frac{1}{\alpha(t)}\mathbf{d}\tran(t)[\mathbf{x} -\mathbf{x}(t-1)] + \frac{1}{4\alpha^{2}(t)} \Vert \mathbf{d}(t) \Vert^{2}\Big] \nonumber\\
\overset{(b)}{=}& \argmin_{\mathbf{x}\in \mathcal{X}} \left[ \mathbf{d}\tran(t)\mathbf{x}+  \alpha(t)\Vert \mathbf{x} -\mathbf{x}(t-1) \Vert^{2} \right] \label{eq:lm-projection-update-smooth-opt},
\end{align}
where (a) follows from the definition of the projection onto a convex set; and (b) follows from the fact the minimizing solution does not change when we remove constant term $-\frac{1}{\alpha(t)} \mathbf{d}\tran(t) \mathbf{x}(t-1) + \frac{1}{4\alpha^{2}(t)} \Vert \mathbf{d}(t)\Vert^{2}$ and multiply positive constant $\alpha(t)$ in the objective function.

Recall that $\mathbf{d}(t) = \nabla f(\mathbf{x}(t-1)) + \sum_{k=1}^{m} [Q_{k}(t) + g_{k}(\mathbf{x}(t-1))] \nabla g_{k}(\mathbf{x}(t-1))$. This lemma follows because \eqref{eq:lm-projection-update-smooth-opt} is identical to the update of $\mathbf{x}(t)$ in \cref{alg:new-alg} when $\tilde{f}(\mathbf{x})\equiv 0$ and $\tilde{\mathbf{g}}(\mathbf{x})\equiv \mathbf{0}$. 
\end{proof}

For constrained smooth convex programs, \cref{lm:x-update-proj-implementation} suggests that \cref{alg:new-alg} has a similar per-iteration complexity when compared with \cref{alg:primal-dual-subgradient}. However, \cref{alg:new-alg} can be more easily implemented since it does not require any upper bound of $\boldsymbol{\lambda}^{\ast}$ as required by \cref{alg:primal-dual-subgradient}. Moreover, we shall show that \cref{alg:general-alg} has   faster $O(1/\epsilon)$ convergence in comparison with the slow $O(1/\epsilon^2)$ convergence of \cref{alg:primal-dual-subgradient}.

\section{Preliminaries and Basis Analysis}

This section presents useful preliminaries on convex analysis and important facts of \cref{alg:new-alg}.

\subsection{Preliminaries}

\begin{Def}[Lipschitz Continuity] \label{def:Lipschitz-continuous}
Let $\mathcal{X} \subseteq \mathbb{R}^n$ be a convex set. Function $\phi: \mathcal{X}\rightarrow \mathbb{R}^m$ is said to be Lipschitz continuous  on $\mathcal{X}$ with modulus $L$ if there exists $L> 0$ such that $\Vert \phi(\mathbf{y}) - \phi(\mathbf{x}) \Vert \leq L \Vert\mathbf{y} - \mathbf{x}\Vert$  for all $ \mathbf{x}, \mathbf{y} \in \mathcal{X}$. 
\end{Def}

\begin{Def}[Smooth Functions]
Let $\mathcal{X} \subseteq \mathbb{R}^n$ and function $\phi(\mathbf{x})$ be continuously differentiable on $\mathcal{X}$. Function $\phi(\mathbf{x})$ is said to be smooth on $\mathcal{X}$ with modulus $L$ if $\nabla \phi(\mathbf{x})$ is Lipschitz continuous on $\mathcal{X}$ with modulus $L$.
\end{Def}

Note that linear function $\phi(\mathbf{x}) = \mathbf{a}^T \mathbf{x}$ is smooth with modulus $0$.  If a function $\phi(\mathbf{x})$ is smooth with modulus $L$, then $c \phi(\mathbf{x})$ is smooth with modulus $cL$ for any $c>0$.

\begin{Lem}[Descent Lemma, Proposition A.24 in \cite{book_NonlinearProgramming_Bertsekas}] \label{lm:descent-lemma} If $h$ is smooth on $\mathcal{X}$ with modulus $L$, then $\phi(\mathbf{y}) \leq \phi(\mathbf{x}) + \nabla \phi(\mathbf{x})^T (\mathbf{y} - \mathbf{x}) + \frac{L}{2} \Vert \mathbf{y} - \mathbf{x}\Vert^2$ for all $\mathbf{x}, \mathbf{y} \in \mathcal{X}$.
\end{Lem}

\begin{Def}[Strongly Convex Functions]
 Let $\mathcal{X} \subseteq \mathbb{R}^n$ be a convex set. Function $\phi$ is said to be strongly convex on $\mathcal{X}$ with modulus $\alpha$ if there exists a constant $\alpha>0$ such that $\phi(\mathbf{x}) - \frac{1}{2} \alpha \Vert \mathbf{x} \Vert^2$ is convex on $\mathcal{X}$.
\end{Def}

By the definition of strongly convex functions, it is easy to show that if $\phi(\mathbf{x})$ is convex and $\alpha>0$, then $\phi(\mathbf{x}) + \alpha \Vert \mathbf{x} - \mathbf{x}_0\Vert^2$ is strongly convex with modulus $2\alpha$ for any constant $\mathbf{x}_0$.

\begin{Lem} [Corollary 1 in \cite{YuNeely17SIOPT}] \label{lm:strong-convex-quadratic-optimality}
Let $\mathcal{X} \subseteq \mathbb{R}^{n}$ be a convex set. Let function $\phi$ be strongly convex on $\mathcal{X}$ with modulus $\alpha$ and $\mathbf{x}^{opt}$ be a global minimum of $h$ on $\mathcal{X}$. Then, $\phi(\mathbf{x}^{opt}) \leq \phi(\mathbf{x}) - \frac{\alpha}{2} \Vert \mathbf{x}^{opt} - \mathbf{x}\Vert^{2}$ for all $\mathbf{x}\in \mathcal{X}$.
\end{Lem}

\subsection{Properties of the Virtual Queue Vector and the Drift}

The following preliminary results (Lemmas \ref{lm:virtual-queue}-\ref{lm:queue-constraint-inequality}) on virtual queue vector $\mathbf{Q}(t)$ and its drift are proven for \cref{alg:general-alg} in \cite{YuNeely17SIOPT} and hold regardless of the update of $\mathbf{x}(t)$.  Since  \cref{alg:new-alg} has the same update equation of $\mathbf{Q}(t)$,  these results also hold for \cref{alg:new-alg}.

\begin{Lem} [Lemma 3 in \cite{YuNeely17SIOPT}]\label{lm:virtual-queue} In \cref{alg:new-alg}, we have
\begin{enumerate}
\item At each iteration $t\in\{0,1,2,\ldots\}$, $Q_k(t)\geq 0$ for all $k\in\{1,2,\ldots,m\}$.
\item At each iteration $t\in\{0,1,2,\ldots\}$, $Q_{k}(t) + G_{k}(\mathbf{x}(t-1))\geq 0$ for all $k\in\{1,2\ldots, m\}$.
\item At iteration $t=0$, $\Vert \mathbf{Q}(0)\Vert^2 \leq \Vert \mathbf{G}(\mathbf{x}(-1))\Vert^2$. At each iteration $t\in\{1,2,\ldots\}$,  $\Vert \mathbf{Q}(t)\Vert^2 \geq \Vert \mathbf{G}(\mathbf{x}(t-1))\Vert^2$.
\end{enumerate}
\end{Lem}

\begin{Lem}[Lemma 7 in \cite{YuNeely17SIOPT}]\label{lm:queue-constraint-inequality}
Let $\mathbf{Q}(t), t\in\{0,1,\ldots\}$ be the sequence generated by \cref{alg:new-alg}.  
For any $t\geq 1$, 
\[ Q_k(t) \geq   \displaystyle{\sum_{\tau=0}^{t-1} G_k(\mathbf{x}(\tau))}, \forall k\in\{1,2,\ldots,m\}. \]
\end{Lem}

Let $\mathbf{Q}(t) = \big[ Q_1(t), \ldots, Q_m(t)\big]\tran$ be the vector of virtual queue backlogs.  Define  $L(t) = \frac{1}{2} \Vert \mathbf{Q}(t)\Vert^2$. The function $L(t)$ shall be called a \emph{Lyapunov function}. Define the {Lyapunov drift} as 
\begin{align}
\Delta (t) = L(t+1) - L(t) = \frac{1}{2} [ \Vert \mathbf{Q}(t+1)\Vert^{2} - \Vert \mathbf{Q}(t)\Vert^{2}]. \label{eq:def-drift}
\end{align}

\begin{Lem} [Lemma 4 in \cite{YuNeely17SIOPT}]\label{lm:drift} At each iteration $t\in\{0,1,2,\ldots\}$ in \cref{alg:new-alg}, an upper bound of the Lyapunov drift is given by
\begin{align}
\Delta(t) \leq \mathbf{Q}\tran(t) \mathbf{G}(\mathbf{x}(t))  +\Vert \mathbf{G}(\mathbf{x}(t))\Vert^2.  \label{eq:drift}
\end{align}
\end{Lem}

\subsection{Properties from Strong Duality}

The next lemma follows from \cref{lm:queue-constraint-inequality} and \cref{as:strong-duality}.

\begin{Lem}\label{lm:obj-diff-bound-from-strong-duality}
Let $\mathbf{x}^{\ast}$ be an optimal solution of problem \eqref{eq:program-objective}-\eqref{eq:program-set-constraint} and $\boldsymbol{\lambda}^\ast$ be a Lagrange multiplier vector satisfying \cref{as:strong-duality}. Let $\mathbf{x}(t), \mathbf{Q}(t), t\in\{0,1,\ldots\}$ be sequences generated by \cref{alg:new-alg}. Then,
\begin{align*}
\sum_{\tau=0}^{t-1} F(\mathbf{x}(\tau)) \geq t F(\mathbf{x}^\ast) -  \Vert \boldsymbol{\lambda}^\ast\Vert \Vert \mathbf{Q}(t)\Vert, \quad \forall t\geq 1. 
\end{align*}
\end{Lem}

\begin{proof}
The proof is quite similar to the proof of Lemma 8 in \cite{YuNeely17SIOPT}. Define Lagrangian dual function $q(\boldsymbol{\lambda}) = \min\limits_{\mathbf{x}\in \mathcal{X}}\{F(\mathbf{x})+ \sum_{k=1}^m \lambda_k G_k(\mathbf{x})\}$. For all $\tau\in\{0,1,\ldots\}$, by \cref{as:strong-duality}, we have
\begin{align*}
F(\mathbf{x}^{\ast}) = q(\boldsymbol{\lambda}^{\ast}) \overset{(a)}{\leq}  F(\mathbf{x}(\tau)) + \sum_{k=1}^m \lambda_k^\ast G_k(\mathbf{x}(\tau)),
\end{align*} 
where (a) follows the definition of $q(\boldsymbol{\lambda}^{\ast})$.

Thus, we have $ F(\mathbf{x}(\tau))  \geq F(\mathbf{x}^\ast) -\sum_{k=1}^m \lambda_k^\ast G_k(\mathbf{x}(\tau)),\forall \tau \in\{0,1,\ldots\}$. Summing over $\tau\in \{0,1,\ldots, t-1\}$ yields 
\begin{align}
\sum_{\tau=0}^{t-1} F(\mathbf{x}(\tau)) \geq &t F(\mathbf{x}^\ast) -  \sum_{\tau=0}^{t-1} \sum_{k=1}^m\lambda_k^\ast G_k(\mathbf{x}(\tau))\nonumber \\
=&t F(\mathbf{x}^\ast)  - \sum_{k=1}^m  \lambda_k^\ast \Big[\sum_{\tau=0}^{t-1}G_k(\mathbf{x}(\tau))\Big] \nonumber\\
\overset{(a)}{\geq}& t F(\mathbf{x}^\ast)  -\sum_{k=1}^m  \lambda_k^\ast Q_k(t) \nonumber\\
\overset{(b)}{\geq}& t F(\mathbf{x}^\ast)  -\Vert \boldsymbol{\lambda}^\ast\Vert _{2}\Vert \mathbf{Q}(t)\Vert, \nonumber
\end{align}
where $(a)$ follows from \cref{lm:queue-constraint-inequality} and the fact that $\lambda_k^\ast \geq 0, \forall k\in\{1,2,\ldots, m\}$; and $(b)$ follows from the Cauchy-Schwarz inequality. 
\end{proof}

\subsection{An Upper Bound of the Drift-Plus-Penalty Expression}

\begin{Lem}\label{lm:dpp-bound}
Let $\mathbf{x}^{\ast}$ be an optimal solution of problem \eqref{eq:program-objective}-\eqref{eq:program-set-constraint}. For all $t\geq 0$ in \cref{alg:new-alg}, we have
\begin{align*}
&\Delta(t) + F(\mathbf{x}(t)) \\
\leq &F(\mathbf{x}^{\ast}) + \alpha(t) \big[\Vert \mathbf{x}^{\ast}- \mathbf{x}(t-1)\Vert^{2} - \Vert \mathbf{x}^{\ast} - \mathbf{x}(t)\Vert^{2}\big] +\frac{1}{2} \big[ \Vert \mathbf{G}(\mathbf{x}(t))\Vert^{2} -  \Vert \mathbf{G}(\mathbf{x}(t-1))\Vert^{2} \big] \\ &+\big[\frac{\beta^{2} + L_{f} + [\mathbf{Q}(t) + \mathbf{G}(\mathbf{x}(t-1))]\tran \mathbf{L}_{\mathbf{g}}}{2} -\alpha(t) \big] \Vert \mathbf{x}(t) - \mathbf{x}(t-1)\Vert^{2}
\end{align*}
where $\beta, L_{f}$ and $\mathbf{L}_{\mathbf{g}}$ are defined in \cref{as:basic}.
\end{Lem}

\begin{proof} 
Fix $t\geq 0$. By part (2) in \cref{lm:virtual-queue}, $Q_{k}(t) + G_{k}(\mathbf{x}(t-1))\geq 0, \forall k\in\{1,2,\ldots,m\}$. Thus, $[\nabla f(\mathbf{x}(t-1))]\tran \mathbf{x}  + \tilde{f}(\mathbf{x}) + \sum_{k=1}^m \big[Q_k(t) + G_k(\mathbf{x}(t-1))\big]\big[[\nabla g_k(\mathbf{x}(t-1))]\tran \mathbf{x} + \tilde{g}_k(\mathbf{x})\big]$ is convex with respect to $\mathbf{x}\in \mathcal{X}$. Since $\alpha(t) \Vert \mathbf{x} - \mathbf{x}(t-1)\Vert^{2}$ is strongly convex with respect to $\mathbf{x}$ with modulus $2\alpha(t)$, it follows that  $[\nabla f(\mathbf{x}(t-1))]\tran \mathbf{x}  + \tilde{f}(\mathbf{x}) + \sum_{k=1}^m \big[Q_k(t) + G_k(\mathbf{x}(t-1))\big]\big[[\nabla g_k(\mathbf{x}(t-1))]\tran \mathbf{x} + \tilde{g}_k(\mathbf{x})\big]  + \alpha(t) \Vert \mathbf{x} - \mathbf{x}(t-1)\Vert^{2}$ is strongly convex with respect to $\mathbf{x}$ with modulus $2\alpha(t)$.

Since $\mathbf{x}(t)$ is chosen to minimize the above strongly convex function, by \cref{lm:strong-convex-quadratic-optimality},  we have

\begin{align*}
& [\nabla f(\mathbf{x}(t-1))]\tran \mathbf{x}(t)  + \tilde{f}(\mathbf{x}(t))  \\
&+ \sum_{k=1}^m \big[Q_k(t) + G_k(\mathbf{x}(t-1))\big]\big[[\nabla g_k(\mathbf{x}(t-1))]\tran \mathbf{x}(t) + \tilde{g}_k(\mathbf{x}(t))\big]  \\ &+ \alpha(t) \Vert \mathbf{x}(t) - \mathbf{x}(t-1)\Vert^{2}\\
\leq & [\nabla f(\mathbf{x}(t-1))]\tran \mathbf{x}^{\ast}  + \tilde{f}(\mathbf{x}^\ast) \\ &+ \sum_{k=1}^m \big[Q_k(t) + G_k(\mathbf{x}(t-1))\big]\big[[\nabla g_k(\mathbf{x}(t-1))]\tran \mathbf{x}^{\ast} + \tilde{g}_k(\mathbf{x}^{\ast})\big]  \\ &+ \alpha(t) \Vert \mathbf{x}^{\ast} - \mathbf{x}(t-1)\Vert^{2} - \alpha(t) \Vert \mathbf{x}^{\ast} - \mathbf{x}(t)\Vert^{2}
\end{align*}

Adding constant $f(\mathbf{x}(t-1)) - [\nabla f(\mathbf{x}(t-1))]\tran \mathbf{x}(t-1)+ \sum_{k=1}^{m}\big[Q_{k}(t) + G_{k}(t-1)\big]\big[g_{k}(\mathbf{x}(t-1)) - [\nabla g_{k}(\mathbf{x}(t-1))]\tran \mathbf{x}(t-1)\big]$ on both sides and rearranging terms yields
\begin{align}
& f(\mathbf{x}(t-1)) + [\nabla f(\mathbf{x}(t-1))]\tran [\mathbf{x}(t)-\mathbf{x}(t-1)]  + \tilde{f}(\mathbf{x}(t))  \nonumber\\ & + \sum_{k=1}^m \big[Q_k(t) + G_k(\mathbf{x}(t-1))\big]\Big[g_{k}(\mathbf{x}(t-1)) + [\nabla g_k(\mathbf{x}(t-1))]\tran [\mathbf{x}(t)-\mathbf{x}(t-1)] \nonumber\\& \qquad\qquad\qquad\qquad\qquad\qquad\quad+ \tilde{g}_k(\mathbf{x}(t))\Big]  \nonumber\\
\leq & f(\mathbf{x}(t-1)) + [\nabla f(\mathbf{x}(t-1))]\tran [\mathbf{x}^{\ast} - \mathbf{x}(t-1)]  + \tilde{f}(\mathbf{x}^{\ast})  \nonumber\\ &+ \sum_{k=1}^m \big[Q_k(t) + G_k(\mathbf{x}(t-1))\big]\Big[ g_{k}(\mathbf{x}(t-1)) + [\nabla g_k(\mathbf{x}(t-1))]\tran[ \mathbf{x}^{\ast}-\mathbf{x}(t-1)] \nonumber \\&\qquad\qquad\qquad\qquad\qquad\qquad\quad+ \tilde{g}_k(\mathbf{x}^{\ast})\Big]   \nonumber\\ &\quad -  \alpha(t) \Vert \mathbf{x}(t) - \mathbf{x}(t-1)\Vert^{2}  + \alpha(t) \big[\Vert \mathbf{x}^{\ast} - \mathbf{x}(t-1)\Vert^{2} - \Vert \mathbf{x}^{\ast} - \mathbf{x}(t)\Vert^{2}\big]  \nonumber\\
\overset{(a)}{\leq}& f(\mathbf{x}^{\ast}) + \tilde{f}(\mathbf{x}^{\ast}) + \sum_{k=1}^m \big[Q_k(t) + G_k(\mathbf{x}(t-1))\big]\big[ g_{k}(\mathbf{x}^{\ast}) + \tilde{g}_k(\mathbf{x}^{\ast})\big]  \nonumber\\ &- \alpha(t) \Vert \mathbf{x}(t) - \mathbf{x}(t-1)\Vert^{2}   + \alpha(t) \big[\Vert \mathbf{x}^{\ast} - \mathbf{x}(t-1)\Vert^{2} - \Vert \mathbf{x}^{\ast} - \mathbf{x}(t)\Vert^{2}\big]  \nonumber\\
\overset{(b)}{\leq} &F(\mathbf{x}^{\ast})  -  \alpha(t) \Vert \mathbf{x}(t) - \mathbf{x}(t-1)\Vert^{2}  + \alpha(t) \big[\Vert \mathbf{x}^{\ast} - \mathbf{x}(t-1)\Vert^{2} - \Vert \mathbf{x}^{\ast} - \mathbf{x}(t)\Vert^{2}\big] \label{eq:pf-dpp-bound-eq1}
\end{align}
where (a) follows from the convexity of $f(\cdot)$ and each $g_{k}(\cdot)$, and the fact that $Q_{k}(t) + G_{k}(\mathbf{x}(t-1))\geq 0, \forall k\in\{1,2,\ldots,m\}$ (i.e., part (2) in \cref{lm:virtual-queue}); and (b) follows because $F(\mathbf{x}^\ast)=f(\mathbf{x}^\ast)+\tilde{f}(\mathbf{x}^\ast)$, $G_{k}(\mathbf{x}^{\ast}) = g_{k}(\mathbf{x}^{\ast}) + \tilde{g}_{k}(\mathbf{x}^{\ast})\leq 0, \forall k\in\{1,2,\ldots,m\}$, which further follows from the feasibility of $\mathbf{x}^{\ast}$, and $Q_{k}(t) + G_{k}(\mathbf{x}(t-1))\geq 0, \forall k\in\{1,2,\ldots,m\}$ (i.e., part (2) in \cref{lm:virtual-queue}).

Recall that $f(\mathbf{x})$ is smooth on $\mathcal{X}$ with modulus $L_{f}$ by \cref{as:basic}. By \cref{lm:descent-lemma}, we have
\begin{align*}
f(\mathbf{x}(t))  \leq f(\mathbf{x}(t-1)) + [\nabla f(\mathbf{x}(t-1))]\tran[\mathbf{x}(t) - \mathbf{x}(t-1)] +\frac{L_{f}}{2} \Vert \mathbf{x}(t) - \mathbf{x}(t-1)\Vert^{2}.
\end{align*}
Adding $\tilde{f}(\mathbf{x}(t))$ on both sides, recalling $F(\mathbf{x}) = f(\mathbf{x})+ \tilde{f}(\mathbf{x})$ and rearranging terms yields
\begin{align}
&F(\mathbf{x}(t))-\frac{L_{f}}{2} \Vert \mathbf{x}(t) - \mathbf{x}(t-1)\Vert^{2} \nonumber \\
\leq &f(\mathbf{x}(t-1)) + [\nabla f(\mathbf{x}(t-1))]\tran[\mathbf{x}(t) - \mathbf{x}(t-1)]   + \tilde{f}(\mathbf{x}(t)). \label{eq:pf-dpp-bound-eq2}
\end{align}

Recall that each $g_{k}(\mathbf{x})$ is smooth on $\mathcal{X}$ with modulus $L_{g_{k}}$ by \cref{as:basic}. Thus, each $[Q_{k}(t) + G_{k}(\mathbf{x}(t-1))] g_{k}(\mathbf{x})$ is smooth with modulus $[Q_{k}(t) + G_{k}(\mathbf{x}(t-1))] L_{g_{k}}$. By \cref{lm:descent-lemma}, we have
\begin{align}
&[Q_{k}(t) + G_{k}(\mathbf{x}(t-1)) ] g_{k}(\mathbf{x}(t)) \nonumber \\
\leq &[Q_{k}(t) + G_{k}(\mathbf{x}(t-1)) ] g_{k}(\mathbf{x}(t-1)) \nonumber \\&+ [Q_{k}(t) + G_{k}(\mathbf{x}(t-1))] [\nabla g_{k}(\mathbf{x}(t-1))]\tran  [\mathbf{x}(t) - \mathbf{x}(t-1)] \nonumber \\ &+  \frac{[Q_{k}(t) + G_{k}(\mathbf{x}(t-1))] L_{g_{k}}}{2} \Vert \mathbf{x}(t) - \mathbf{x}(t-1)\Vert^{2} . \label{eq:pf-dpp-bound-eq3}
\end{align}
Summing \eqref{eq:pf-dpp-bound-eq3} over $k\in\{1,2,\ldots, m\}$ yields
\begin{align*}
&\sum_{k=1}^{m} [Q_{k}(t) + G_{k}(\mathbf{x}(t-1)) ] g_{k}(\mathbf{x}(t))\\
\leq &\sum_{k=1}^{m}\big[Q_{k}(t) + G_{k}(\mathbf{x}(t-1))\big] \big[ g_{k}(\mathbf{x}(t-1)) + [\nabla g_{k}(\mathbf{x}(t-1))]\tran [\mathbf{x}(t) - \mathbf{x}(t-1)] \big]\nonumber \\ &+  \frac{[\mathbf{Q}(t) + \mathbf{G}(\mathbf{x}(t-1))]\tran \mathbf{L}_{\mathbf{g}}}{2} \Vert \mathbf{x}(t) - \mathbf{x}(t-1)\Vert^{2}. 
\end{align*}
Adding $\sum_{k=1}^{m}[Q_{k}(t) + G_{k}(\mathbf{x}(t-1)) ] \tilde{g}_{k}(\mathbf{x}(t))$ on both sides, recalling $G_k(\mathbf{x}(t)) = g_k(\mathbf{x}(t)) + \tilde{g}_k(\mathbf{x}(t))$ and rearranging terms yields
\begin{align}
&\sum_{k=1}^{m} [Q_{k}(t) + G_{k}(\mathbf{x}(t-1)) ] G_{k}(\mathbf{x}(t)) -  \frac{[\mathbf{Q}(t) + \mathbf{G}(\mathbf{x}(t-1))]\tran \mathbf{L}_{\mathbf{g}}}{2} \Vert \mathbf{x}(t) - \mathbf{x}(t-1)\Vert^{2} \nonumber\\
\leq &\sum_{k=1}^{m}\big[Q_{k}(t) + G_{k}(\mathbf{x}(t-1))\big] \Big[ g_{k}(\mathbf{x}(t-1)) + [\nabla g_{k}(\mathbf{x}(t-1))]\tran [\mathbf{x}(t) - \mathbf{x}(t-1)] \nonumber\\& \qquad\qquad\qquad\qquad\qquad\qquad+ \tilde{g}_{k}(\mathbf{x}(t))\Big]. \label{eq:pf-dpp-bound-eq4} 
\end{align}

Summing up \eqref{eq:pf-dpp-bound-eq2} and \eqref{eq:pf-dpp-bound-eq4} together yields
\begin{align}
&F(\mathbf{x}(t)) + [\mathbf{Q}(t) + \mathbf{G}(\mathbf{x}(t-1))]\tran \mathbf{G}(\mathbf{x}(t)) \nonumber\\& - \frac{L_{f} + [\mathbf{Q}(t) + \mathbf{G}(\mathbf{x}(t-1))]\tran \mathbf{L}_{\mathbf{g}}}{2} \Vert \mathbf{x}(t) - \mathbf{x}(t-1)\Vert^{2}  \nonumber \\
\leq & f(\mathbf{x}(t-1)) + [\nabla f(\mathbf{x}(t-1))]\tran [\mathbf{x}(t)-\mathbf{x}(t-1)]  + \tilde{f}(\mathbf{x}(t))  \nonumber\\ &+\sum_{k=1}^m \big[Q_k(t) + G_k(\mathbf{x}(t-1))\big]\Big[g_{k}(\mathbf{x}(t-1)) + [\nabla g_k(\mathbf{x}(t-1))]\tran [\mathbf{x}(t)-\mathbf{x}(t-1)] \nonumber\\& \qquad\qquad\qquad\qquad\qquad\qquad\quad+ \tilde{g}_k(\mathbf{x}(t))\Big].  \label{eq:pf-dpp-bound-eq5}
\end{align}

Note that the right side of \eqref{eq:pf-dpp-bound-eq5}  is identical to the left side of \eqref{eq:pf-dpp-bound-eq1}. Thus, by combining \eqref{eq:pf-dpp-bound-eq1} and \eqref{eq:pf-dpp-bound-eq5}; and rearranging terms, we have 
 \begin{align}
&F(\mathbf{x}(t)) + [\mathbf{Q}(t) + \mathbf{G}(\mathbf{x}(t-1))]\tran \mathbf{G}(\mathbf{x}(t))   \nonumber \\
\leq &F(\mathbf{x}^{\ast}) + \alpha(t) \big[\Vert \mathbf{x}^{\ast} - \mathbf{x}(t-1)\Vert^{2} - \Vert \mathbf{x}^{\ast} - \mathbf{x}(t)\Vert^{2}\big] \nonumber \\&+ \big[\frac{L_{f} + [\mathbf{Q}(t) + \mathbf{G}(\mathbf{x}(t-1))]\tran \mathbf{L}_{\mathbf{g}}}{2} - \alpha(t)\big] \Vert \mathbf{x}(t) - \mathbf{x}(t-1)\Vert^{2}. \label{eq:pf-dpp-bound-eq6}
\end{align}
 Note that $\mathbf{u} \tran \mathbf{v}= \frac{1}{2} [\Vert  \mathbf{u}\Vert^{2}  + \Vert \mathbf{v}\Vert^{2} - \Vert \mathbf{u} - \mathbf{v}\Vert^{2} ]$ for any $\mathbf{u}, \mathbf{v}\in \mathbb{R}^{m}$. Thus, we have 
 \begin{align}
&[\mathbf{G}(\mathbf{x}(t-1))]\tran \mathbf{G}(\mathbf{x}(t))  \nonumber\\ =&  \frac{1}{2} [ \Vert \mathbf{G}(\mathbf{x}(t-1))\Vert^{2}  + \Vert \mathbf{G}(\mathbf{x}(t))\Vert^{2} - \Vert \mathbf{G}(\mathbf{x}(t-1))-\mathbf{G}(\mathbf{x}(t))\Vert^{2} ]. \label{eq:pf-dpp-bound-eq7} 
\end{align} 
Substituting \eqref{eq:pf-dpp-bound-eq7} into \eqref{eq:pf-dpp-bound-eq6} and rearranging terms  yields
\begin{align*}
 &F(\mathbf{x}(t)) + \mathbf{Q}\tran(t)\mathbf{G}(\mathbf{x}(t)) \nonumber \\
 \leq & F(\mathbf{x}^{\ast}) + \alpha(t) \big[\Vert \mathbf{x}^{\ast}- \mathbf{x}(t-1)\Vert^{2} - \Vert \mathbf{x}^{\ast} - \mathbf{x}(t)\Vert^{2}\big]  \nonumber\\ &+ \big[\frac{L_{f} + [\mathbf{Q}(t) + \mathbf{G}(\mathbf{x}(t-1))]\tran \mathbf{L}_{\mathbf{g}}}{2} -\alpha(t) \big] \Vert \mathbf{x}(t) - \mathbf{x}(t-1)\Vert^{2}   \\ &+ \frac{1}{2}  \Vert \mathbf{G}(\mathbf{x}(t-1))-\mathbf{G}(\mathbf{x}(t))\Vert^{2} - \frac{1}{2}  \Vert \mathbf{G}(\mathbf{x}(t-1))\Vert^{2} - \frac{1}{2}  \Vert \mathbf{G}(\mathbf{x}(t))\Vert^{2}\\
 \overset{(a)}{\leq} &F(\mathbf{x}^{\ast}) + \alpha(t) \big[\Vert \mathbf{x}^{\ast}- \mathbf{x}(t-1)\Vert^{2} -\Vert \mathbf{x}^{\ast} - \mathbf{x}(t)\Vert^{2}\big] \nonumber \\&+ \big[\frac{\beta^{2} + L_{f} + [\mathbf{Q}(t) + \mathbf{G}(\mathbf{x}(t-1))]\tran \mathbf{L}_{\mathbf{g}}}{2} -\alpha(t) \big] \Vert \mathbf{x}(t) - \mathbf{x}(t-1)\Vert^{2} \\ &- \frac{1}{2}  \Vert \mathbf{G}(\mathbf{x}(t-1))\Vert^{2} - \frac{1}{2}  \Vert \mathbf{G}(\mathbf{x}(t))\Vert^{2},
 \end{align*}
where (a) follows from the fact that $\Vert \mathbf{G}(\mathbf{x}(t-1)) - \mathbf{G}(\mathbf{x}(t))\Vert \leq \beta \Vert \mathbf{x}(t) - \mathbf{x}(t-1)\Vert$, which further follows from the assumption that $\mathbf{G}(\mathbf{x})$ is Lipschitz continuous with modulus $\beta$.

Summing \eqref{eq:drift} to the above inequality yields
\begin{align*}
&\Delta(t) + F(\mathbf{x}(t)) \\
\leq &F(\mathbf{x}^{\ast}) + \alpha(t) \big[\Vert \mathbf{x}^{\ast}- \mathbf{x}(t-1)\Vert^{2} - \Vert \mathbf{x}^{\ast} - \mathbf{x}(t)\Vert^{2}\big] +\frac{1}{2} \big[ \Vert \mathbf{G}(\mathbf{x}(t))\Vert^{2} -  \Vert \mathbf{G}(\mathbf{x}(t-1))\Vert^{2} \big] \\ &+\big[\frac{\beta^{2} + L_{f} + [\mathbf{Q}(t) + \mathbf{G}(\mathbf{x}(t-1))]\tran \mathbf{L}_{\mathbf{g}}}{2} -\alpha(t) \big] \Vert \mathbf{x}(t) - \mathbf{x}(t-1)\Vert^{2}
\end{align*}
\end{proof}

The next corollary follows directly by noting that $\mathbf{L}_{\mathbf{g}} = \mathbf{0}$ when each $g_k(\mathbf{x})$ is a linear function.

\begin{Cor}\label{cor:dpp-bound-linear}
Let $\mathbf{x}^{\ast}$ be an optimal solution of problem \eqref{eq:program-objective}-\eqref{eq:program-set-constraint} where each $g_k(\mathbf{x})$ is a linear function. If $\alpha(t) =  \alpha > \frac{1}{2}[\beta^2 + L_f], \forall t\geq 0$ in \cref{alg:new-alg}, then for all $t\geq 0$, we have
\begin{align*}
&\Delta(t) + F(\mathbf{x}(t)) \\
\leq &F(\mathbf{x}^{\ast}) + \alpha \big[\Vert \mathbf{x}^{\ast}- \mathbf{x}(t-1)\Vert^{2} - \Vert \mathbf{x}^{\ast} - \mathbf{x}(t)\Vert^{2}\big] +\frac{1}{2} \big[ \Vert \mathbf{G}(\mathbf{x}(t))\Vert^{2} -  \Vert \mathbf{G}(\mathbf{x}(t-1))\Vert^{2} \big] 
\end{align*}
where $\beta$ and $L_{f}$  are defined in \cref{as:basic}.
\end{Cor}

\begin{proof}
Note that if each $g_k(\mathbf{x})$ is a linear function, then we have $\mathbf{L}_{\mathbf{g}} = \mathbf{0}$. Fix $t\geq 0$. By \cref{lm:dpp-bound} with $\alpha(t) = \alpha$ and $\mathbf{L}_{\mathbf{g}} = \mathbf{0}$, we have
\begin{align*}
&\Delta(t) + F(\mathbf{x}(t)) \\
\leq &F(\mathbf{x}^{\ast}) + \alpha \big[\Vert \mathbf{x}^{\ast}- \mathbf{x}(t-1)\Vert^{2} - \Vert \mathbf{x}^{\ast} - \mathbf{x}(t)\Vert^{2}\big] +\frac{1}{2} \big[ \Vert \mathbf{G}(\mathbf{x}(t))\Vert^{2} -  \Vert \mathbf{G}(\mathbf{x}(t-1))\Vert^{2} \big] \\
& +\big[\frac{\beta^{2} + L_{f} }{2} -\alpha \big] \Vert \mathbf{x}(t) - \mathbf{x}(t-1)\Vert^{2}\\
\overset{(a)}{\leq} &F(\mathbf{x}^{\ast}) + \alpha \big[\Vert \mathbf{x}^{\ast}- \mathbf{x}(t-1)\Vert^{2} - \Vert \mathbf{x}^{\ast} - \mathbf{x}(t)\Vert^{2}\big] +\frac{1}{2} \big[ \Vert \mathbf{G}(\mathbf{x}(t))\Vert^{2} -  \Vert \mathbf{G}(\mathbf{x}(t-1))\Vert^{2} \big]
\end{align*}
where (a) follows from $\alpha >  \frac{1}{2}[\beta^2 + L_f]$.
\end{proof}

\section{Convergence Time Analysis of \cref{alg:new-alg} } \label{sec:convergence-time}

This section analyzes the convergence time of \cref{alg:new-alg} for convex program \eqref{eq:program-objective}-\eqref{eq:program-set-constraint}. In particular, the following two rules for choosing $\alpha(t)$ in \cref{alg:new-alg} are considered.
\begin{itemize}
\item Constant $\alpha(t)$: Choose algorithm parameters $\alpha(t)$ via
\begin{align}
\alpha(t) = \alpha >  \frac{1}{2}[\beta^2 + L_f], \forall t\geq 0 \label{eq:constant-alpha}
\end{align}
\item Non-decreasing $\alpha(t)$:  Choose algorithm parameters $\alpha(t)$ via
\begin{align}
\alpha(t) = \left\{ \begin{array}{ll}\frac{1}{2}\big[\beta^2+L_f + [\mathbf{Q}(0) + \mathbf{G}(\mathbf{x}(-1))]\tran \mathbf{L}_{\mathbf{g}} \big], & t=0\\ \max\big\{\alpha(t-1), \frac{1}{2}\big[\beta^2+L_f + [\mathbf{Q}(t) + \mathbf{G}(\mathbf{x}(t-1))]\tran \mathbf{L}_{\mathbf{g}} \big]\big\}, &t\geq 1\end{array}\right. \label{eq:increasing-alpha}
\end{align}
Note that part (2) of \cref{lm:virtual-queue} implies $\alpha(0)>0$, and hence 
$\alpha(t)>0, \forall t\geq 0$ since $\alpha(t)$ is a nondecreasing sequence.  
\end{itemize}

\subsection{Convex Programs with Linear $\mathbf{g}(\mathbf{x})$}
This subsection proves that if each $g_k(\mathbf{x})$ is a linear function, then it suffices to choose constant parameters $\alpha(t) = \alpha > \frac{1}{2}[\beta^2 + L_f]$ in \cref{alg:new-alg} to solve convex program \eqref{eq:program-objective}-\eqref{eq:program-set-constraint} with an $O(1/\epsilon)$ convergence time.

\begin{Thm} \label{thm:performance-linear}
Consider convex program \eqref{eq:program-objective}-\eqref{eq:program-set-constraint} under \cref{as:basic} and \cref{as:strong-duality} where each $g_k(\mathbf{x})$ is a linear function. Let $\mathbf{x}^\ast$ be an optimal solution. Let $\boldsymbol{\lambda}^\ast$ be a Lagrange multiplier vector satisfying \cref{as:strong-duality}. If we choose constant $\alpha(t)$ in \cref{alg:new-alg} according to \eqref{eq:constant-alpha}, then for all $t\geq 1$, we have
\begin{enumerate}
\item $F(\overline{\mathbf{x}}(t))  \leq  F(\mathbf{x}^{\ast}) + \frac{\alpha}{t}\Vert \mathbf{x}^{\ast} - \mathbf{x}(-1)\Vert^{2}$.
\item $G_k(\overline{\mathbf{x}}(t)) \leq \frac{1}{t} \big[ \Vert \boldsymbol{\lambda}^{\ast}\Vert + \sqrt{2\alpha} \Vert \mathbf{x}^\ast - \mathbf{x}(t-1)\Vert + \sqrt{\frac{\alpha}{\alpha-\frac{1}{2}\beta^2 - \frac{1}{2}L_f}}\Vert \mathbf{G}(\mathbf{x}^\ast)\Vert \big]$.
\end{enumerate}
where $\beta$ and $L_f$ are defined in \cref{as:basic}. That is, \cref{alg:new-alg} ensures error decays like $O(1/t)$ and provides an $\epsilon$-approximate solution with convergence time $O(1/\epsilon)$.
\end{Thm}
\begin{proof}~ 
\begin{enumerate}
\item By \cref{cor:dpp-bound-linear}, we have $F(\mathbf{x}(\tau))\leq F(\mathbf{x}^\ast) +\alpha \big[\Vert \mathbf{x}^{\ast}- \mathbf{x}(t-1)\Vert^{2} - \Vert \mathbf{x}^{\ast} - \mathbf{x}(\tau)\Vert^{2}\big]+\frac{1}{2} \big[ \Vert \mathbf{G}(\mathbf{x}(\tau))\Vert^{2} -  \Vert \mathbf{G}(\mathbf{x}(\tau-1))\Vert^{2} \big] - \Delta(\tau)$ for all $\tau\in\{0,1,2,\ldots\}$. Fix $t\geq 1$.  Summing over $\tau\in\{0,1,2,\ldots, t-1\}$ yields
\begin{align}
&\sum_{\tau=0}^{t-1}F(\mathbf{x}(\tau)) \nonumber\\
\leq& t F(\mathbf{x}^\ast) + \alpha \sum_{\tau=0}^{t-1}\big[\Vert \mathbf{x}^{\ast}- \mathbf{x}(\tau-1)\Vert^{2} - \Vert \mathbf{x}^{\ast} - \mathbf{x}(\tau)\Vert^{2}\big] \nonumber\\ &+\frac{1}{2} \sum_{\tau=0}^{t-1}\big[ \Vert \mathbf{G}(\mathbf{x}(\tau))\Vert^{2} -  \Vert \mathbf{G}(\mathbf{x}(\tau-1))\Vert^{2} \big] - \sum_{\tau=0}^{t-1}\Delta(\tau) \nonumber\\
\overset{(a)}{=}& t F(\mathbf{x}^\ast) + \alpha \Vert \mathbf{x}^{\ast}- \mathbf{x}(-1)\Vert^{2} -\alpha \Vert \mathbf{x}^{\ast} - \mathbf{x}(t-1)\Vert^{2} \nonumber\\ &+\frac{1}{2}\Vert \mathbf{G}(\mathbf{x}(t-1))\Vert^{2} -  \Vert \mathbf{G}(\mathbf{x}(-1))\Vert^{2} + \frac{1}{2}\Vert \mathbf{Q}(0)\Vert^2 - \frac{1}{2}\Vert \mathbf{Q}(t)\Vert^2 \nonumber\\
\overset{(b)}{\leq}&t F(\mathbf{x}^\ast) + \alpha \Vert \mathbf{x}^{\ast}- \mathbf{x}(-1)\Vert^{2} -\alpha \Vert \mathbf{x}^{\ast} - \mathbf{x}(t-1)\Vert^{2} +\frac{1}{2}\Vert \mathbf{G}(\mathbf{x}(t-1))\Vert^{2}  - \frac{1}{2}\Vert \mathbf{Q}(t)\Vert^2 \label{eq:pf-thm-performance-linear-eq1}
\end{align}
where (a) follows by recalling that $\Delta(\tau) =\frac{1}{2}\Vert \mathbf{Q}(\tau+1)\Vert^{2} - \frac{1}{2}\Vert \mathbf{Q}(\tau)\Vert^{2} $ and (b) follows from $\Vert \mathbf{Q}(0)\Vert^2 \leq \Vert \mathbf{G}(\mathbf{x}(-1))\Vert^2$ by part (3) in \cref{lm:virtual-queue}.

Recalling that $\Vert \mathbf{Q}(t)\Vert^2 \geq \Vert \mathbf{G}(\mathbf{x}(t-1))\Vert^2$ by part (3) in \cref{lm:virtual-queue} and ignoring a negative term $-\alpha \Vert \mathbf{x}^{\ast} - \mathbf{x}(t-1)\Vert^{2} $ on the right side of \eqref{eq:pf-thm-performance-linear-eq1} yields
\begin{align*}
\sum_{\tau=0}^{t-1}F(\mathbf{x}(\tau)) \leq t F(\mathbf{x}^\ast) + \alpha \Vert \mathbf{x}^{\ast}- \mathbf{x}(-1)\Vert^{2}
\end{align*}
Dividing both sides by $t$ and using Jensen's inequality for convex function $F(\mathbf{x})$ yields
\begin{align*}
F(\overline{\mathbf{x}}(t))  \leq  F(\mathbf{x}^{\ast}) + \frac{\alpha}{t}\Vert \mathbf{x}^{\ast} - \mathbf{x}(-1)\Vert^{2}.
\end{align*}

\item Fix $t\geq 1$. Note that \eqref{eq:pf-thm-performance-linear-eq1} can be written as
\begin{align}
&\sum_{\tau=0}^{t-1}F(\mathbf{x}(\tau)) \nonumber\\
\leq &t F(\mathbf{x}^\ast) + \alpha \Vert \mathbf{x}^{\ast}- \mathbf{x}(-1)\Vert^{2} -\alpha \Vert \mathbf{x}^{\ast} - \mathbf{x}(t-1)\Vert^{2} \nonumber \\ &+\frac{1}{2}\Vert \mathbf{G}(\mathbf{x}(t-1)) - \mathbf{G}(\mathbf{x}^\ast)+ \mathbf{G}(\mathbf{x}^\ast)\Vert^{2}  - \frac{1}{2}\Vert \mathbf{Q}(t)\Vert^2 \nonumber \\
=& t F(\mathbf{x}^\ast) + \alpha \Vert \mathbf{x}^{\ast}- \mathbf{x}(-1)\Vert^{2} -\alpha \Vert \mathbf{x}^{\ast} - \mathbf{x}(t-1)\Vert^{2} +\frac{1}{2}\Vert \mathbf{G}(\mathbf{x}(t-1)) - \mathbf{G}(\mathbf{x}^\ast)\Vert^{2}\nonumber \\ &+[\mathbf{G}(\mathbf{x}^\ast)]\tran [\mathbf{G}(\mathbf{x}(t-1)) - \mathbf{G}(\mathbf{x}^\ast)] +\frac{1}{2} \Vert \mathbf{G}(\mathbf{x}^\ast)\Vert^2 - \frac{1}{2}\Vert \mathbf{Q}(t)\Vert^2 \nonumber
\end{align}
\begin{align}
\overset{(a)}{\leq}& t F(\mathbf{x}^\ast) + \alpha \Vert \mathbf{x}^{\ast}- \mathbf{x}(-1)\Vert^{2} -\alpha \Vert \mathbf{x}^{\ast} - \mathbf{x}(t-1)\Vert^{2} +\frac{1}{2}\Vert \mathbf{G}(\mathbf{x}(t-1)) - \mathbf{G}(\mathbf{x}^\ast)\Vert^{2}\nonumber \\ &+\Vert \mathbf{G}(\mathbf{x}^\ast) \Vert \Vert\mathbf{G}(\mathbf{x}(t-1)) - \mathbf{G}(\mathbf{x}^\ast)\Vert  +\frac{1}{2} \Vert \mathbf{G}(\mathbf{x}^\ast)\Vert^2- \frac{1}{2}\Vert \mathbf{Q}(t)\Vert^2 \nonumber \\
\overset{(b)}{\leq} & t F(\mathbf{x}^\ast) + \alpha \Vert \mathbf{x}^{\ast}- \mathbf{x}(-1)\Vert^{2} -\alpha \Vert \mathbf{x}^{\ast} - \mathbf{x}(t-1)\Vert^{2} +\frac{1}{2}\beta^2\Vert \mathbf{x}^\ast - \mathbf{x}(t-1)\Vert^{2}\nonumber \\ &+\beta \Vert \mathbf{G}(\mathbf{x}^\ast) \Vert \Vert \mathbf{x}^\ast - \mathbf{x}(t-1)\Vert +\frac{1}{2} \Vert \mathbf{G}(\mathbf{x}^\ast)\Vert^2- \frac{1}{2}\Vert \mathbf{Q}(t)\Vert^2 \nonumber \\
=& t F(\mathbf{x}^\ast) + \alpha \Vert \mathbf{x}^{\ast}- \mathbf{x}(-1)\Vert^{2} - [\alpha - \frac{1}{2}\beta^2] \Big[ \Vert \mathbf{x}^\ast - \mathbf{x}(t-1)\Vert - \frac{\beta}{2\alpha-\beta^2} \Vert \mathbf{G}(\mathbf{x}^\ast)\Vert \Big]^2 \nonumber\\ &+\frac{\alpha}{2\alpha-\beta^2} \Vert \mathbf{G}(\mathbf{x}^\ast) \Vert^2 - \frac{1}{2}\Vert \mathbf{Q}(t)\Vert^2 \nonumber\\
\overset{(c)}{\leq}&t F(\mathbf{x}^\ast) + \alpha \Vert \mathbf{x}^{\ast}- \mathbf{x}(-1)\Vert^{2} +\frac{\alpha}{2\alpha-\beta^2} \Vert \mathbf{G}(\mathbf{x}^\ast) \Vert^2 - \frac{1}{2}\Vert \mathbf{Q}(t)\Vert^2 \label{eq:pf-thm-performance-linear-eq2}
\end{align}
where (a) follows from the Cauchy-Schwarz inequality; (b) follows from Lipschitz continuity of $\mathbf{G}(\mathbf{x})$ in \cref{as:basic}; and (c) follows from $\alpha > \frac{1}{2}[\beta^2 + L_f] \geq \frac{1}{2}\beta^2$.

By \cref{lm:obj-diff-bound-from-strong-duality}, we have
\begin{align}
\sum_{\tau=0}^{t-1} F(\mathbf{x}(\tau)) \geq t F(\mathbf{x}^\ast) -  \Vert \boldsymbol{\lambda}^\ast\Vert \Vert \mathbf{Q}(t)\Vert\label{eq:pf-thm-performance-linear-eq3}
\end{align}
Combining \eqref{eq:pf-thm-performance-linear-eq2} and \eqref{eq:pf-thm-performance-linear-eq3}, cancelling common terms and rearranging terms yields
\begin{align}
&\frac{1}{2}\Vert \mathbf{Q}(t)\Vert^2  -\Vert \boldsymbol{\lambda}^\ast\Vert \Vert \mathbf{Q}(t)\Vert - \alpha \Vert \mathbf{x}^{\ast}- \mathbf{x}(-1)\Vert^{2} -\frac{\alpha}{2\alpha-\beta^2} \Vert \mathbf{G}(\mathbf{x}^\ast) \Vert^2 \leq 0 \nonumber\\
\Rightarrow & \Big [ \Vert \mathbf{Q}(t)\Vert - \Vert \boldsymbol{\lambda}^\ast\Vert\Big]^2\leq \Vert \boldsymbol{\lambda}^\ast\Vert^2 + 2\alpha \Vert \mathbf{x}^{\ast}- \mathbf{x}(-1)\Vert^{2} +\frac{\alpha}{\alpha-\frac{1}{2}\beta^2} \Vert \mathbf{G}(\mathbf{x}^\ast) \Vert^2 \nonumber\\
\Rightarrow &\Vert \mathbf{Q}(t)\Vert \leq \Vert \boldsymbol{\lambda}^\ast\Vert +\sqrt{ \Vert \boldsymbol{\lambda}^\ast\Vert^2 + 2\alpha \Vert \mathbf{x}^{\ast}- \mathbf{x}(-1)\Vert^{2} +\frac{\alpha}{\alpha-\frac{1}{2}\beta^2} \Vert \mathbf{G}(\mathbf{x}^\ast) \Vert^2} \nonumber\\
\overset{(a)}{\Rightarrow}&\Vert \mathbf{Q}(t)\Vert \leq 2\Vert \boldsymbol{\lambda}^\ast\Vert + \sqrt{2\alpha} \Vert \mathbf{x}^{\ast}- \mathbf{x}(-1)\Vert +\sqrt{\frac{\alpha}{\alpha-\frac{1}{2}\beta^2}} \Vert \mathbf{G}(\mathbf{x}^\ast) \Vert \label{eq:pf-thm-performance-linear-eq4}
\end{align}
where (a) follows from the basic inequality $\sqrt{z_1 + z_2 + z_3}\leq \sqrt{z_1} + \sqrt{z_1} + \sqrt{z_3}$ for any $z_1, z_2, z_3\geq0$.

Fix $k\in\{1,2,\ldots, m\}$. By Jensen's inequality for convex function $G_k(\mathbf{x})$, we have
\begin{align*}
G_k(\overline{\mathbf{x}}(t)) \leq& \frac{1}{t}\sum_{\tau=0}^{t-1} G_k(\mathbf{x}(\tau))\\
\overset{(a)}{\leq}& \frac{1}{t}Q_k(t)
\end{align*}
\begin{align*}
\leq & \frac{1}{t} \Vert \mathbf{Q}(t)\Vert\\
\overset{(b)}{\leq} & \frac{1}{t} \Big[ 2\Vert \boldsymbol{\lambda}^\ast\Vert + \sqrt{2\alpha} \Vert \mathbf{x}^{\ast}- \mathbf{x}(-1)\Vert +\sqrt{\frac{\alpha}{\alpha-\frac{1}{2}\beta^2}} \Vert \mathbf{G}(\mathbf{x}^\ast) \Vert\Big]
\end{align*}
where (a) follows from \cref{lm:queue-constraint-inequality} and (b) follows from \eqref{eq:pf-thm-performance-linear-eq4}.
\end{enumerate}
\end{proof}

\subsection{Convex Programs with Possibly Non-linear $\mathbf{g}(\mathbf{x})$}
For convex program \eqref{eq:program-objective}-\eqref{eq:program-set-constraint} with possibly nonlinear $\mathbf{g}(\mathbf{x})$, the following assumption is further assumed:

\begin{Assumption} \label{as:boundedness}~
\begin{itemize}
\item There exists $C>0$ such that $\Vert \mathbf{G}(\mathbf{x}) \Vert \leq C$ for all $\mathbf{x}\in \mathcal{X}$.  
\item There exists $R>0$ such that $\Vert \mathbf{x} - \mathbf{y}\Vert\leq R$ for all $\mathbf{x}, \mathbf{y}\in \mathcal{X}$.
\end{itemize}
\end{Assumption}
Note that \cref{as:boundedness} holds when $\mathcal{X}$ is a compact set. 

This subsection proves that if convex program \eqref{eq:program-objective}-\eqref{eq:program-set-constraint} with possibly nonlinear $\mathbf{g}(\mathbf{x})$ satisfies Assumptions \ref{as:basic}-\ref{as:boundedness}, then it suffices to choose non-decreasing parameters $\alpha(t)$ according to \eqref{eq:increasing-alpha} in \cref{alg:new-alg} to solve convex program \eqref{eq:program-objective}-\eqref{eq:program-set-constraint} with an $O(1/\epsilon)$ convergence time.

\begin{Lem}\label{lm:non-linear-g}
Consider convex program \eqref{eq:program-objective}-\eqref{eq:program-set-constraint} under Assumptions \ref{as:basic}-\ref{as:boundedness}.  Let $\mathbf{x}^\ast$ be an optimal solution and  $\boldsymbol{\lambda}^\ast$ be a Lagrange multiplier vector satisfying \cref{as:strong-duality}. If we choose non-decreasing $\alpha(t)$ in \cref{alg:new-alg} according to \eqref{eq:increasing-alpha}, then we have
\begin{enumerate}
\item $\sum_{\tau=0}^{t} \alpha(\tau) \big[ \Vert \mathbf{x}^{\ast}- \mathbf{x}(\tau-1)\Vert^{2} - \Vert \mathbf{x}^{\ast} - \mathbf{x}(\tau)\Vert^{2}\big] \leq \alpha(t) R^2, \forall t\geq 0$;
\item $\sum_{\tau=0}^{t-1}F(\mathbf{x}(\tau)) \leq t F(\mathbf{x}^\ast) +\alpha(t-1)R^2 + \frac{1}{2}\Vert \mathbf{G}(\mathbf{x}(t-1))\Vert^{2}  - \frac{1}{2}\Vert \mathbf{Q}(t)\Vert^2, \forall t\geq 1$;
\item $\Vert \mathbf{Q}(t+1)\Vert \leq 2\Vert \boldsymbol{\lambda}^\ast\Vert + R\sqrt{2\alpha(t)} + C, \forall t\geq 0$;
\end{enumerate}
where  $R$ and $C$ are defined in \cref{as:boundedness}.
\end{Lem}
\begin{proof}~
\begin{enumerate}
\item  This is obviously true when $t=0$. Fix $t\geq 1$. Note that
\begin{align*}
&\sum_{\tau=0}^{t} \alpha(\tau) \big[ \Vert \mathbf{x}^{\ast}- \mathbf{x}(\tau-1)\Vert^{2} - \Vert \mathbf{x}^{\ast} - \mathbf{x}(\tau)\Vert^{2}\big] \\
=&\alpha(0) \Vert \mathbf{x}^\ast - \mathbf{x}(-1)\Vert^2 + \sum_{\tau=0}^{t-1} [\alpha(\tau+1) - \alpha(\tau)] \Vert \mathbf{x}^\ast - \mathbf{x}(\tau)\Vert^2 - \alpha(t) \Vert \mathbf{x}^\ast - \mathbf{x}(t)\Vert^2\\
\overset{(a)}{\leq}& \alpha(0) R^2 + \sum_{\tau=0}^{t-1} [\alpha(\tau+1) - \alpha(\tau)]  R^2\\
=& \alpha(t)R^2
\end{align*}
where (a) follows because $\Vert \mathbf{x}^\ast - \mathbf{x}(\tau)\Vert\leq R, \forall \tau\geq 0$ by \cref{as:boundedness} and $\alpha(\tau+1) \geq \alpha(\tau), \forall \tau\geq 0$ by \eqref{eq:increasing-alpha}.
\item Fix $t\geq 1$.  By \cref{lm:dpp-bound}, for all $\tau\in\{0,1,2,\ldots\}$, we have
\begin{align*}
&\Delta(\tau) + F(\mathbf{x}(\tau)) \\
\leq &F(\mathbf{x}^{\ast}) + \alpha(\tau) \big[\Vert \mathbf{x}^{\ast}- \mathbf{x}(\tau-1)\Vert^{2} - \Vert \mathbf{x}^{\ast} - \mathbf{x}(\tau)\Vert^{2}\big] \\ &+\frac{1}{2} \big[ \Vert \mathbf{G}(\mathbf{x}(\tau))\Vert^{2} -  \Vert \mathbf{G}(\mathbf{x}(\tau-1))\Vert^{2} \big]\\
&+\big[\frac{\beta^{2} + L_{f} + [\mathbf{Q}(\tau) + \mathbf{G}(\mathbf{x}(\tau-1))]\tran \mathbf{L}_{\mathbf{g}}}{2} -\alpha(\tau) \big] \Vert \mathbf{x}(\tau) - \mathbf{x}(\tau-1)\Vert^{2}
\end{align*}
\begin{align*}
\overset{(a)}{\leq}&  F(\mathbf{x}^{\ast}) + \alpha(\tau) \big[\Vert \mathbf{x}^{\ast}- \mathbf{x}(\tau-1)\Vert^{2} - \Vert \mathbf{x}^{\ast} - \mathbf{x}(\tau)\Vert^{2}\big] \\ &+\frac{1}{2} \big[ \Vert \mathbf{G}(\mathbf{x}(\tau))\Vert^{2} -  \Vert \mathbf{G}(\mathbf{x}(\tau-1))\Vert^{2} \big]
\end{align*}
where (a) follows because each $\alpha(\tau)$ is chosen to guarantee $\frac{1}{2}[\beta^{2} + L_{f} + [\mathbf{Q}(\tau) + \mathbf{G}(\mathbf{x}(\tau-1))]\tran \mathbf{L}_{\mathbf{g}}] -\alpha(\tau) \leq 0.$

Summing over $\tau\in\{0,1,2,\ldots, t-1\}$ and rearranging terms yields
\begin{align}
&\sum_{\tau=0}^{t-1}F(\mathbf{x}(\tau)) \nonumber\\
\leq& t F(\mathbf{x}^\ast) +  \sum_{\tau=0}^{t-1}\alpha(\tau)\big[\Vert \mathbf{x}^{\ast}- \mathbf{x}(\tau-1)\Vert^{2} - \Vert \mathbf{x}^{\ast} - \mathbf{x}(\tau)\Vert^{2}\big] \nonumber\\ &+\frac{1}{2} \sum_{\tau=0}^{t-1}\big[ \Vert \mathbf{G}(\mathbf{x}(\tau))\Vert^{2} -  \Vert \mathbf{G}(\mathbf{x}(\tau-1))\Vert^{2} \big] - \sum_{\tau=0}^{t-1}\Delta(\tau) \nonumber\\
\overset{(a)}{\leq}& t F(\mathbf{x}^\ast) + \alpha(t-1)R^2 +\frac{1}{2}\Vert \mathbf{G}(\mathbf{x}(t-1))\Vert^{2} - \frac{1}{2} \Vert \mathbf{G}(\mathbf{x}(-1))\Vert^{2} + \frac{1}{2}\Vert \mathbf{Q}(0)\Vert^2 \nonumber \\&- \frac{1}{2}\Vert \mathbf{Q}(t)\Vert^2 \nonumber\\
\overset{(b)}{\leq}& t F(\mathbf{x}^\ast) +\alpha(t-1)R^2 + \frac{1}{2}\Vert \mathbf{G}(\mathbf{x}(t-1))\Vert^{2}  - \frac{1}{2}\Vert \mathbf{Q}(t)\Vert^2 \nonumber
\end{align}
where (a) follows from part (1) of this lemma and by recalling that $\Delta(\tau) =\frac{1}{2}\Vert \mathbf{Q}(\tau+1)\Vert^{2} - \frac{1}{2}\Vert \mathbf{Q}(\tau)\Vert^{2} $; and (b) follows because $\Vert \mathbf{Q}(0)\Vert^2 \leq \Vert \mathbf{G}(\mathbf{x}(-1))\Vert^2$ by part (3) in \cref{lm:virtual-queue}.
\item 
By part (2) of this lemma, we have 
\begin{align}
\sum_{\tau=0}^{t}F(\mathbf{x}(\tau)) \leq& (t+1) F(\mathbf{x}^\ast) +\alpha(t)R^2 +\frac{1}{2}\Vert \mathbf{G}(\mathbf{x}(t))\Vert^{2}  - \frac{1}{2}\Vert \mathbf{Q}(t+1)\Vert^2 \nonumber\\
\leq&(t+1) F(\mathbf{x}^\ast) +\alpha(t)R^2 +\frac{1}{2}C^{2}  - \frac{1}{2}\Vert \mathbf{Q}(t+1)\Vert^2
\label{eq:pf-lm-nonlinear-g-eq1}
\end{align}
where (a) follows from $\Vert \mathbf{G}(\mathbf{x}(t))\Vert\leq C$ by \cref{as:boundedness}.
By \cref{lm:obj-diff-bound-from-strong-duality}, we have
\begin{align}
\sum_{\tau=0}^{t} F(\mathbf{x}(\tau)) \geq (t+1) F(\mathbf{x}^\ast) -  \Vert \boldsymbol{\lambda}^\ast\Vert \Vert \mathbf{Q}(t+1)\Vert\label{eq:pf-lm-nonlinear-g-eq2}
\end{align}
Combining \eqref{eq:pf-lm-nonlinear-g-eq1} and \eqref{eq:pf-lm-nonlinear-g-eq2}, cancelling common terms and rearranging terms yields
\begin{align}
&\frac{1}{2}\Vert \mathbf{Q}(t+1)\Vert^2  -\Vert \boldsymbol{\lambda}^\ast\Vert \Vert \mathbf{Q}(t+1)\Vert - \alpha(t) R^2 - \frac{1}{2}C^2  \leq 0 \nonumber\\
\Rightarrow & \Big [ \Vert \mathbf{Q}(t+1)\Vert - \Vert \boldsymbol{\lambda}^\ast\Vert\Big]^2\leq \Vert \boldsymbol{\lambda}^\ast\Vert^2 + 2\alpha(t) R^2 + C^2 \nonumber\\
\Rightarrow &\Vert \mathbf{Q}(t+1)\Vert \leq \Vert \boldsymbol{\lambda}^\ast\Vert +\sqrt{ \Vert \boldsymbol{\lambda}^\ast\Vert^2 + 2\alpha(t) R^2 + C^2} \nonumber\\
\overset{(a)}{\Rightarrow}&\Vert \mathbf{Q}(t+1)\Vert \leq 2\Vert \boldsymbol{\lambda}^\ast\Vert + \sqrt{2\alpha(t)} R+C 
\end{align}
where (a) follows from the basic inequality $\sqrt{z_1 + z_2 + z_3}\leq \sqrt{z_1} + \sqrt{z_1} + \sqrt{z_3}$ for any $z_1, z_2, z_3\geq0$.

\end{enumerate}
\end{proof}

\begin{Lem}\label{lm:alpha-bound}~Consider convex program \eqref{eq:program-objective}-\eqref{eq:program-set-constraint} under Assumptions \ref{as:basic}-\ref{as:boundedness}. If we choose non-decreasing $\alpha(t)$ in \cref{alg:new-alg} according to \eqref{eq:increasing-alpha}, then $$\alpha(t)\leq \alpha^{\max}, \forall t\geq 0$$ with constant
\begin{align}
\alpha^{\max} = \Big[\sqrt{\frac{1}{2}\beta^2 + \frac{1}{2}L_f + \Vert \boldsymbol{\lambda}^\ast\Vert \Vert \mathbf{L}_{\mathbf{g}}\Vert +C \Vert \mathbf{L}_{\mathbf{g}}\Vert} + \frac{\sqrt{2}}{2}R\Vert \mathbf{L}_{\mathbf{g}}\Vert \Big]^2 \label{eq:alpha-bound}
\end{align}
where $\beta, L_f$ and $\mathbf{L}_\mathbf{g}$ are defined in \cref{as:basic}; and $R$ and $C$ are defined in \cref{as:boundedness}.
\end{Lem}
\begin{proof} This lemma can be proven by induction as follows.  Note that by \eqref{eq:increasing-alpha}, we have
\begin{align*}
\alpha(0) =& \frac{1}{2} \beta^2 + \frac{1}{2}L_f + \frac{1}{2}[\mathbf{Q}(0) + \mathbf{G}(\mathbf{x}(-1))]\tran \mathbf{L}_{\mathbf{g}}] \\
\overset{(a)}{\leq}&\frac{1}{2}\beta^2 + \frac{1}{2}L_f +  \frac{1}{2}\Vert\mathbf{Q}(0) + \mathbf{G}(\mathbf{x}(-1))\Vert \Vert \mathbf{L}_{\mathbf{g}}\Vert\\
\overset{(b)}{\leq} & \frac{1}{2}\beta^2 + \frac{1}{2}L_f +  C\Vert \mathbf{L}_{\mathbf{g}}\Vert\\
\leq & \alpha^{\max}
\end{align*}
where (a) follows from the Cauchy-Schwarz inequality; and (b) follows from $\Vert\mathbf{Q}(0) + \mathbf{G}(\mathbf{x}(-1))\Vert \leq \Vert\mathbf{Q}(0)\Vert + \Vert\mathbf{G}(\mathbf{x}(-1))\Vert \leq 2\Vert\mathbf{G}(\mathbf{x}(-1))\Vert \leq 2C$ where the second inequality follows from part (3) of \cref{lm:virtual-queue} and the third inequality follows from \cref{as:boundedness}. Thus, we have $\alpha(0)\leq \alpha^{\max}$. 

Now assume $\alpha(t)\leq \alpha^{\max}$ holds for $t= t_0$ and consider $t = t_0 +1$.  By  \eqref{eq:increasing-alpha}, $\alpha(t_0+1)$ is given by  
\begin{align*}
\alpha(t_0+1) = \max\big\{\alpha(t_0), \frac{1}{2}\big[\beta^2+L_f + [\mathbf{Q}(t_0+1) + \mathbf{G}(\mathbf{x}(t_0))]\tran \mathbf{L}_{\mathbf{g}} \big]\big\}
\end{align*}
Since $\alpha(t_0)\leq \alpha^{\max}$ by induction hypothesis, to prove $\alpha(t_0+1) \leq \alpha^{\max}$, it remains to prove 
\begin{align*}
\frac{1}{2}\big[\beta^2+L_f + [\mathbf{Q}(t_0+1) + \mathbf{G}(\mathbf{x}(t_0))]\tran \mathbf{L}_{\mathbf{g}} \big] \leq \alpha^{\max}
\end{align*}
By part (3) of \cref{lm:non-linear-g}, we have
\begin{align}
\Vert \mathbf{Q}(t_0+1)\Vert \leq& 2\Vert \boldsymbol{\lambda}^\ast\Vert + R\sqrt{2\alpha(t_0)} + C\nonumber \\
\overset{(a)}{\leq}& 2\Vert \boldsymbol{\lambda}^\ast\Vert + R\sqrt{2\alpha^{\max}} + C\label{eq:pf-alpha-bound-eq1}
\end{align}
where (a) follows the hypothesis in the induction.
Thus, we have
\begin{align*}
&\frac{1}{2}\big[\beta^2+L_f + [\mathbf{Q}(t_0+1) + \mathbf{G}(\mathbf{x}(t_0))]\tran \mathbf{L}_{\mathbf{g}} \big]\\
\overset{(a)}{\leq} & \frac{1}{2}\beta^2 + \frac{1}{2}L_f +  \frac{1}{2} \Vert\mathbf{Q}(t_0+1) + \mathbf{G}(\mathbf{x}(t_0))\Vert \Vert \mathbf{L}_{\mathbf{g}}\Vert\\
\overset{(b)}{\leq} & \frac{1}{2}\beta^2 + \frac{1}{2}L_f +  \frac{1}{2} \Vert\mathbf{Q}(t_0+1)\Vert\Vert \mathbf{L}_{\mathbf{g}}\Vert +\Vert \mathbf{G}(\mathbf{x}(t_0))\Vert \Vert \mathbf{L}_{\mathbf{g}}\Vert\\
\overset{(c)}{\leq}& \frac{1}{2}\beta^2 + \frac{1}{2}L_f + \frac{1}{2}\big[2\Vert \boldsymbol{\lambda}^\ast\Vert + R\sqrt{2\alpha^{\max}} + C] \Vert \mathbf{L}_{\mathbf{g}}\Vert+ \frac{1}{2}C\Vert \mathbf{L}_{\mathbf{g}}\Vert\\
=& \frac{1}{2}\beta^2 + \frac{1}{2}L_f  + \Vert \boldsymbol{\lambda}^\ast\Vert \Vert \mathbf{L}_{\mathbf{g}}\Vert+ C \Vert \mathbf{L}_{\mathbf{g}}\Vert + \frac{\sqrt{2}}{2}R \Vert \mathbf{L}_{\mathbf{g}}\Vert \sqrt{\alpha^{\max}}\\
\overset{(d)}{=} & \frac{1}{2}\beta^2 + \frac{1}{2}L_f  + \Vert \boldsymbol{\lambda}^\ast\Vert \Vert \mathbf{L}_{\mathbf{g}}\Vert+ C \Vert \mathbf{L}_{\mathbf{g}}\Vert  + \big[\frac{\sqrt{2}}{2}R\Vert \mathbf{L}_{\mathbf{g}}\Vert\big]^2\\ 
& + \frac{\sqrt{2}}{2}R \Vert \mathbf{L}_{\mathbf{g}}\Vert \sqrt{\frac{1}{2}\beta^2 + \frac{1}{2}L_f + \Vert \boldsymbol{\lambda}^\ast\Vert \Vert \mathbf{L}_{\mathbf{g}}\Vert +C \Vert \mathbf{L}_{\mathbf{g}}\Vert}\\
\overset{(e)}{\leq} & \Big[\sqrt{\frac{1}{2}\beta^2 + \frac{1}{2}L_f + \Vert \boldsymbol{\lambda}^\ast\Vert \Vert \mathbf{L}_{\mathbf{g}}\Vert +C \Vert \mathbf{L}_{\mathbf{g}}\Vert} + \frac{\sqrt{2}}{2}R\Vert \mathbf{L}_{\mathbf{g}}\Vert \Big]^2\\
=& \alpha^{\max}
\end{align*}
where (a) follows from the Cauchy-Schwarz inequality; (b) follows from triangle inequality; (c) follows from \eqref{eq:pf-alpha-bound-eq1} and $\Vert \mathbf{G}(\mathbf{x}(t_0))\Vert\leq C$ by \cref{as:boundedness}; (d) follows by substituting $\alpha^{\max} = \Big[\sqrt{\frac{1}{2}\beta^2 + \frac{1}{2}L_f + \Vert \boldsymbol{\lambda}^\ast\Vert \Vert \mathbf{L}_{\mathbf{g}}\Vert +C \Vert \mathbf{L}_{\mathbf{g}}\Vert} + \frac{\sqrt{2}}{2}R\Vert \mathbf{L}_{\mathbf{g}}\Vert \Big]^2$; and (e) follow from the basic inequality $z_1^2 +z_2^2 + z_1 z_2 \leq (z_1+z_2)^2$ for any $z_1, z_2\geq 0$.

Thus, we have $\alpha(t_0+1) \leq \alpha^{\max}$.  This lemma follows by induction.
\end{proof}

The next theorem summarizes the $O(1/\epsilon)$ convergence time of \cref{alg:new-alg} for convex program \eqref{eq:program-objective}-\eqref{eq:program-set-constraint} with possibly nonlinear $g_k(\mathbf{x})$.
\begin{Thm} \label{thm:performance-nonlinear}
Consider convex program \eqref{eq:program-objective}-\eqref{eq:program-set-constraint} under Assumptions \ref{as:basic}- \ref{as:boundedness} with possibly nonlinear $g_k(\mathbf{x})$. Let $\mathbf{x}^\ast$ be an optimal solution and  $\boldsymbol{\lambda}^\ast$ be a Lagrange multiplier vector satisfying \cref{as:strong-duality}. If we choose non-decreasing $\alpha(t)$ in \cref{alg:new-alg} according to \eqref{eq:increasing-alpha}, then for all $t\geq 1$, we have
\begin{enumerate}
\item $F(\overline{\mathbf{x}}(t))  \leq  F(\mathbf{x}^{\ast}) + \frac{\alpha^{\max}}{t} R^{2}$.
\item $G_k(\overline{\mathbf{x}}(t)) \leq \frac{1}{t} \big[ \Vert \boldsymbol{\lambda}^\ast\Vert + R\sqrt{2\alpha^{\max}} +C \big]$.
\end{enumerate}
where $\alpha^{\max}$ is defined in \cref{lm:alpha-bound}; and $R$ and $C$ are defined in \cref{as:boundedness}. That is, \cref{alg:new-alg} ensures error decays like $O(1/t)$ and provides an $\epsilon$-approximate solution with convergence time $O(1/\epsilon)$.
\end{Thm}
\begin{proof}~
\begin{enumerate}
\item  Fix $t\geq 1$. By part (2) of \cref{lm:non-linear-g}, we have
\begin{align*}
\sum_{\tau=0}^{t-1}F(\mathbf{x}(\tau)) \leq& t F(\mathbf{x}^\ast) +\alpha(t-1)R^2 + \frac{1}{2}\Vert \mathbf{G}(\mathbf{x}(t-1))\Vert^{2}  - \frac{1}{2}\Vert \mathbf{Q}(t)\Vert^2\\
\overset{(a)}{\leq}&t F(\mathbf{x}^\ast) +\alpha^{\max}R^2 
\end{align*}
where (a) follows from $\alpha(t-1)\leq \alpha^{\max}$ by \cref{lm:alpha-bound} and $\Vert \mathbf{Q}(t)\Vert\geq \Vert \mathbf{G}(\mathbf{x}(t-1))\Vert $ by \cref{lm:virtual-queue}.

Dividing both sides by $t$ and using Jensen's inequality for convex function $F(\mathbf{x})$ yields $F(\overline{\mathbf{x}}(t))  \leq  F(\mathbf{x}^{\ast}) + \frac{\alpha^{\max}}{t} R^2$.

\item 
Fix $t\geq 1$ and $k\in\{1,2,\ldots,m\}$. Recall that $\overline{\mathbf{x}}(t) = \frac{1}{t}\sum_{\tau=0}^{t-1} \mathbf{x}(\tau)$. Thus, 
\begin{align*}
G_k (\overline{\mathbf{x}}(t)) &\overset{(a)}{\leq} \frac{1}{t} \sum_{\tau=0}^{t-1} G_k(\mathbf{x}(\tau)) \\
 &\overset{(b)}{\leq} \frac{Q_k(t)}{t} \\
 &\leq \frac{\Vert \mathbf{Q}(t)\Vert}{t}\\
 &\overset{(c)}{\leq} \frac{1}{t} \big( 2 \Vert \boldsymbol{\lambda}^\ast \Vert + R\sqrt{2\alpha^{\max}} + C \big),
 \end{align*}
where (a) follows from the convexity of $g_k(\mathbf{x}), k\in\{1,2,\ldots,m\}$ and Jensen's inequality; (b) follows from \cref{lm:queue-constraint-inequality}; and (c) follows because $\Vert \mathbf{Q}(t)\Vert \leq 2\Vert \boldsymbol{\lambda}^\ast\Vert + R\sqrt{2\alpha(t-1)} + C$ by part (3) of \cref{lm:non-linear-g} and $\alpha(t-1) \leq \alpha^{\max}$ by \cref{lm:alpha-bound}.
\end{enumerate}
\end{proof}

\section{Numerical Experiment: Minimum Variance Portfolio with Norm Constraints}

\subsection{Minimum Variance Portfolio with the $l_2$ Norm Constraint}

Consider the following constrained smooth optimization
\begin{align*}
\min~~ &  \mathbf{x}\tran \mathbf{M}\mathbf{x}\\
\text{s.t.} \quad  & \sum_{i=1}^n x_i = 1\\
			 & \Vert \mathbf{x}\Vert^2 \leq b \\
			 & 0\leq x_i\leq 1, \forall i=\{1,2,\ldots, n\}
\end{align*}
where $\mathbf{x}$ is the weight vector of $n$ assets and $\mathbf{M}$ is the correlation matrix of all assets. This problem is known as global minimum variance portfolio under flexible norm constraints (GMV-N)  and the $l_2$-norm constraint $\Vert \mathbf{x}\Vert^2 \leq b$ is imposed to avoid a solution $\mathbf{x}$ that concentrates in low volatility assets.  For example, in the special case maximum decorrelation portfolio, we choose $b = 3/n$ in the $l_2$-norm constraint  \cite{Kremer17AOR}. 

Without loss of optimality, we can replace the equality constraint $\sum_{i=1}^n x_i = 1$ with an inequality constraint $\sum_{i=1}^n x_i \geq 1$ in the above formulation to obtain an equivalent reformulation. \footnote{This is because if we relax $\sum_{i=1}^n x_i = 1$ by $\sum_{i=1}^n x_i \geq 1$, the optimal solution $\mathbf{x}^\ast$ to the relaxed problem must satisfy $\sum_{i=1}^n x_i^\ast = 1$.}  This equivalent reformulation is a special case of problem \eqref{eq:program-objective}-\eqref{eq:program-set-constraint} with $\tilde{f}(\mathbf{x})\equiv 0$ and $\tilde{\mathbf{g}}(\mathbf{x}) \equiv 0$. In general, for any convex programs with a linear equality constraint $h(\mathbf{x})= 0$, we can always replace the equality constraint with two convex inequality constraints $h(\mathbf{x})\leq 0$ and $h(\mathbf{x})\geq 0$; and reformulate the convex programs into the general form \eqref{eq:program-objective}-\eqref{eq:program-set-constraint}.  In fact, if the convex program has a linear equality constraint $h(\mathbf{x})= 0$, we can modify the corresponding virtual queue in \cref{alg:new-alg} as $Q_k(t+1) = Q_k + h(\mathbf{x}(t))$ at each iteration to solve it directly. (This is also a property owned by \cref{alg:general-alg} to solve convex programs with linear equality constraints, see e.g., footnote 2 in \cite{YuNeely17SIOPT}.)

Since $\mathbf{M}$ is not diagonal, the objective function is not separable and hence at each iteration the update of $\mathbf{x}(t)$ in \cref{alg:general-alg} requires to solve an $n$-dimensional set constrained quadratic program, which can have huge complexity when $n$ is large.  In contrast, the update of $\mathbf{x}(t)$ in \cref{alg:new-alg} has a closed form update for each coordinate by \cref{lm:x-update-proj-implementation}.

In the numerical experiment, we take $n=500$, $b=3/n$ and generate correlation matrix $\mathbf{M} = [\text{Diag}(\mathbf{N}\tran \mathbf{N})]^{-1/2}\mathbf{N}\tran \mathbf{N}[\text{Diag}(\mathbf{N}\tran \mathbf{N})]^{-1/2}$ where $N$ is an $n\times n$ matrix follows the standard Gaussian distribution. We run both \cref{alg:general-alg} and \cref{alg:new-alg} with the same initial point $\mathbf{x}(0) = \mathbf{0}$.  \cref{fig:smooth_obj} and \cref{fig:smooth_cons} show that both algorithms have quite similar convergence performance as observed in the zoom-in subfigures.  However, when implementing both algorithms using MATLAB in a PC with a 4 core 2.7GHz Intel i7 CPU and 16GB Memory, each iteration of \cref{alg:new-alg} only takes around $1.5$ milliseconds while each iteration of \cref{alg:general-alg} takes around $270$ milliseconds.  (Note that our implementation uses quadprog in MATLAB to solve the box constrained quadratic program involved in each iteration of \cref{alg:general-alg}.) Thus, \cref{alg:new-alg} is $180$ times faster than  \cref{alg:general-alg} in this example.

\begin{figure}[htbp]
\centering
   \includegraphics[width=0.8\textwidth,height=0.8\textheight,keepaspectratio=true]{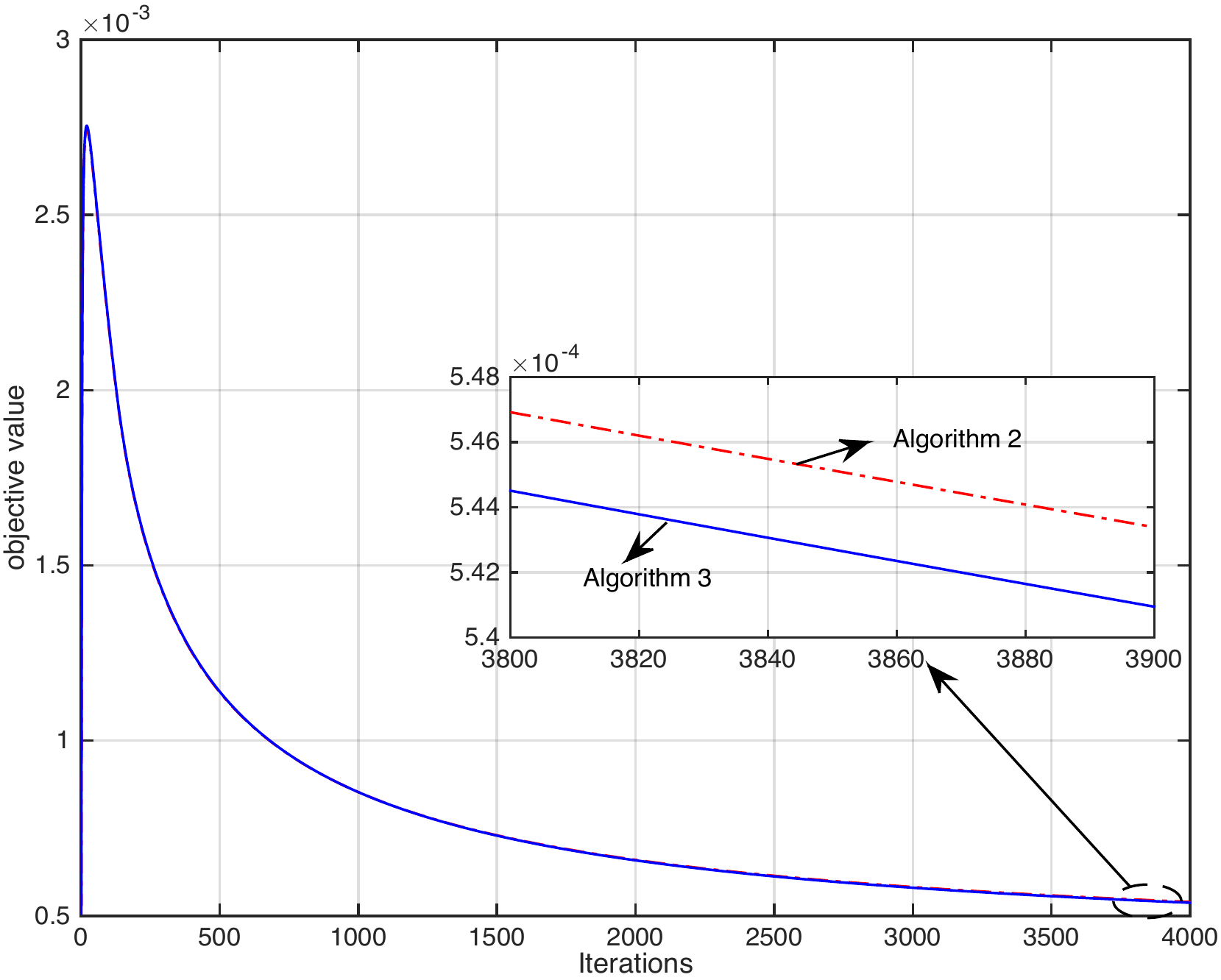} 
   \caption{Minimum variance portfolio with the $l_2$ norm constraint:  objective value performance.}
   \label{fig:smooth_obj}
\end{figure}

\begin{figure}[htbp]
\centering
   \includegraphics[width=0.8\textwidth,height=0.8\textheight,keepaspectratio=true]{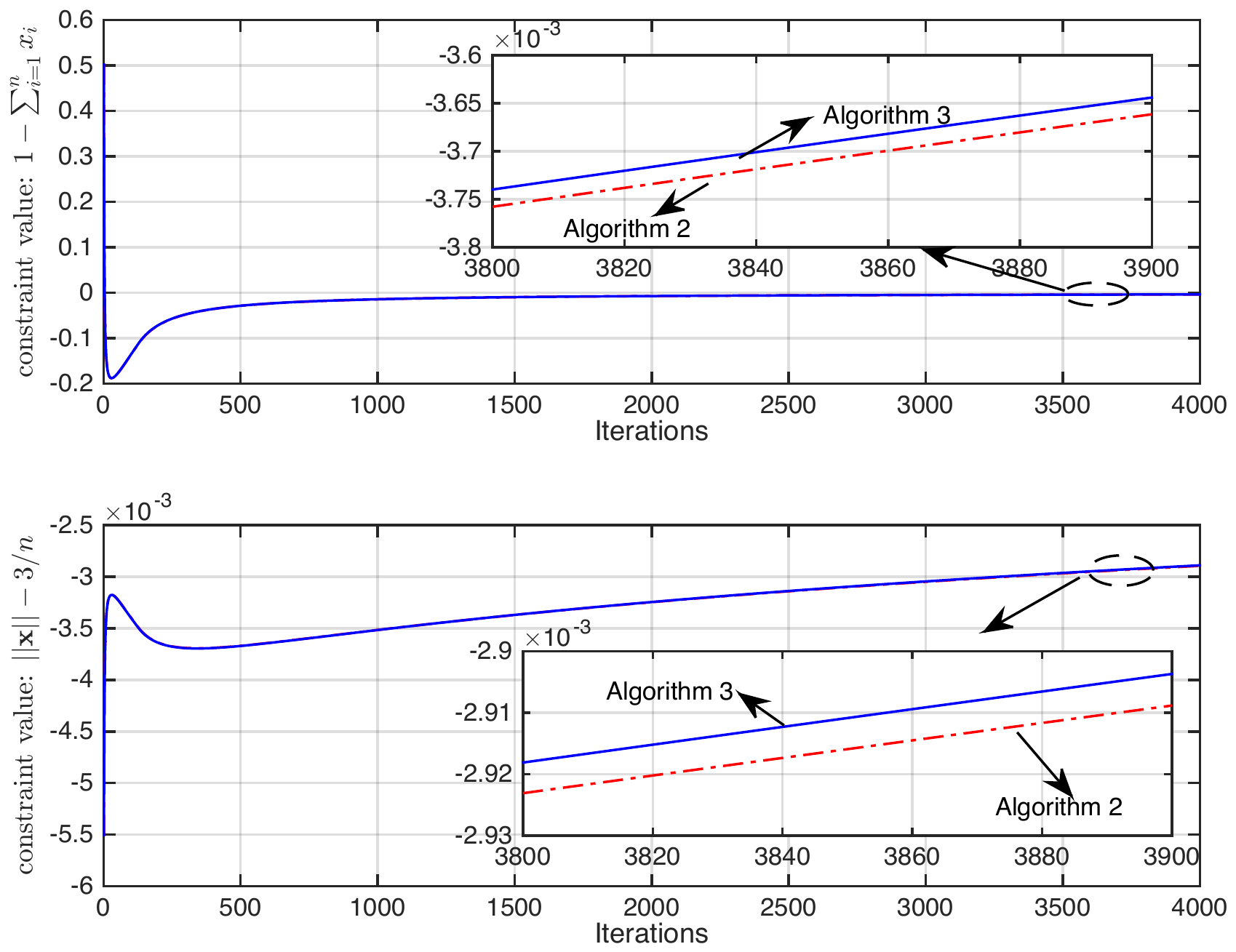} 
   \caption{Minimum variance portfolio with the $l_2$ norm constraint:  constraint value performance.}
   \label{fig:smooth_cons}
\end{figure}

\subsection{Minimum Variance Portfolio with the $l_1$ Norm Constraint}
Consider the following constrained non-smooth optimization
\begin{align*}
\min~~ &  \mathbf{x}\tran \mathbf{M}\mathbf{x}\\
\text{s.t.} \quad  & \sum_{i=1}^n x_i = 1\\
			 & \Vert \mathbf{x}\Vert_1 \leq b
\end{align*}
where $\mathbf{x}$ is the weight vector of $n$ assets and $\mathbf{M}$ is the correlation matrix of all assets. Note that each component $x_i\in \mathbb{R}$ can be possibly negative by assuming that we can sell short the considered assets. The $l_1$ norm constraint $\Vert \mathbf{x}\Vert_1 = \sum_{i=1}^n |x_i|\leq b$ is imposed to promote sparsity  and other desired properties. For example, the minimum variance portfolio with the shortsale constraint considered in \cite{Jagannathan03} is corresponding to the special case $\delta=1$ in the $l_1$ norm constraint \cite{DeMiguel09MS}.

Similarly to the minimum variance portfolio with the $l_2$ norm constraint, we can replace the equality constraint $\sum_{i=1}^n x_i = 1$ with an inequality constraint $\sum_{i=1}^n x_i \geq 1$ in the above formulation to obtain an equivalent reformulation that is special case of problem \eqref{eq:program-objective}-\eqref{eq:program-set-constraint} with $\tilde{f}(\mathbf{x})\equiv 0$, $\mathbf{g}(\mathbf{x}) = [1-\sum_{i=1}^n x_i, 0]\tran$ and $\tilde{\mathbf{g}}(\mathbf{x}) = [0, \sum_{i=1}^n |x_i|- b]\tran$. 

Since $\mathbf{M}$ is not diagonal, the objective function is not separable and hence at each iteration the update of $\mathbf{x}(t)$ in \cref{alg:general-alg} requires to solve an $n$-dimensional unconstrained composite minimization, which can have huge complexity when $n$ is large. In contrast, each iteration of \cref{alg:new-alg} has a closed form update for each coordinate by \cref{lm:l1-norm-closed-form}.

In the numerical experiment, we take $n=500$, $b=3/n$ and generate correlation matrix $\mathbf{M} = [\text{Diag}(\mathbf{N}\tran \mathbf{N})]^{-1/2}\mathbf{N}\tran \mathbf{N}[\text{Diag}(\mathbf{N}\tran \mathbf{N})]^{-1/2}$ where $N$ is an $n\times n$ matrix follows the standard Gaussian distribution. We run both \cref{alg:general-alg} and \cref{alg:new-alg} with the same initial point $\mathbf{x}(0) = \mathbf{0}$.  \cref{fig:nonsmooth_obj} and \cref{fig:nonsmooth_cons} show that both algorithms have quite similar convergence performance as observed in the zoom-in subfigures. However, when implementing both algorithms using MATLAB in a PC with a 4 core 2.7GHz Intel i7 CPU and 16GB memory, each iteration of \cref{alg:new-alg} only takes around $1.5$ milliseconds while each iteration of \cref{alg:general-alg} takes around $2.7$ seconds. (Note that our implementation uses CVX \cite{CVX} to solve the unconstrained composite minimization involved in each iteration of \cref{alg:general-alg}.) Thus, \cref{alg:new-alg} is $1800$ times faster than  \cref{alg:general-alg} in this example.

\begin{figure}[htbp]
\centering
   \includegraphics[width=0.8\textwidth,height=0.8\textheight,keepaspectratio=true]{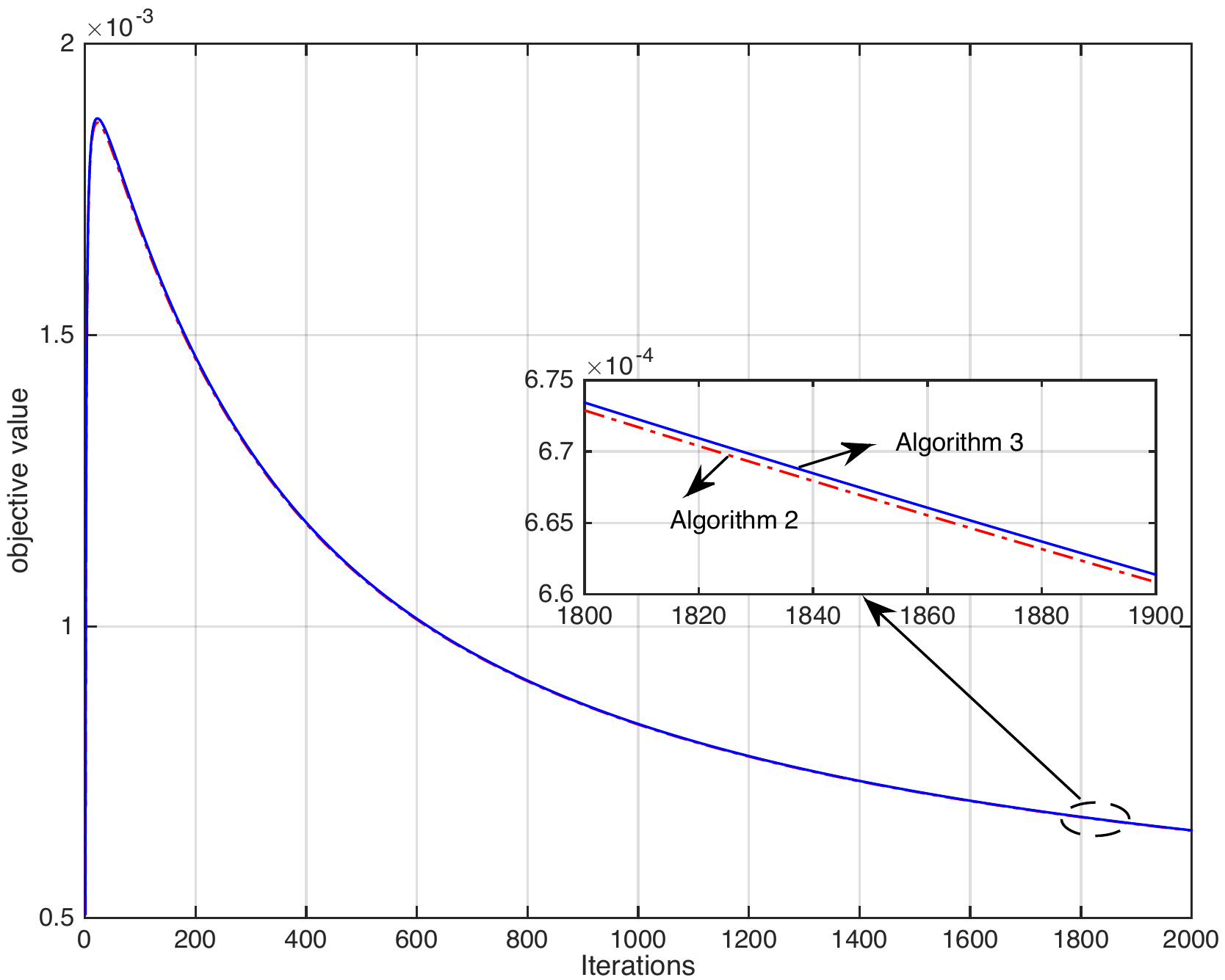} 
   \caption{Minimum variance portfolio with the $l_1$ norm constraint:  objective value performance.}
   \label{fig:nonsmooth_obj}
\end{figure}

\begin{figure}[htbp]
\centering
   \includegraphics[width=0.8\textwidth,height=0.8\textheight,keepaspectratio=true]{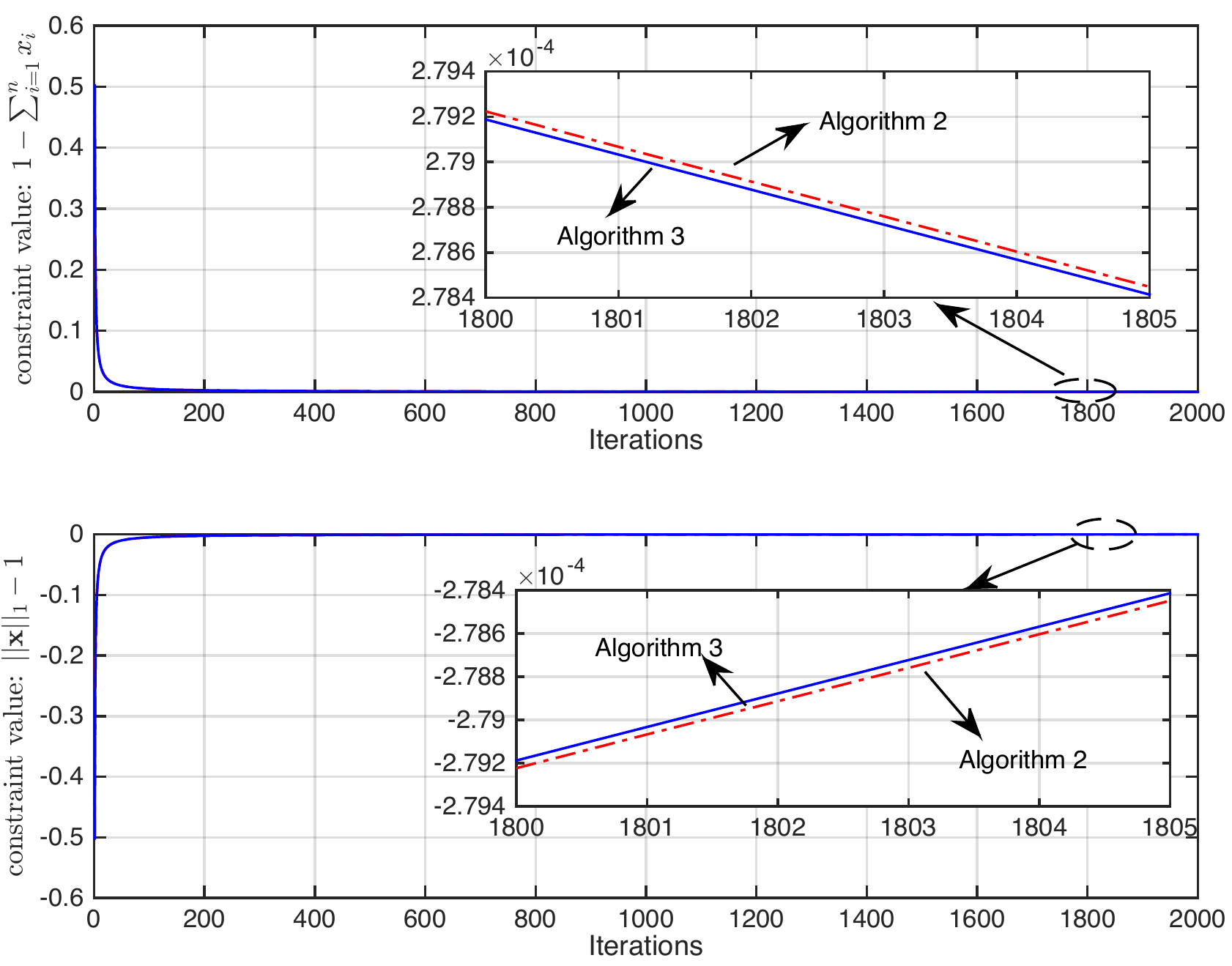} 
   \caption{Minimum variance portfolio with the $l_1$ norm constraint:  constraint value performance.}
   \label{fig:nonsmooth_cons}
\end{figure}

\section{Conclusion}

This paper proposes a new primal-dual type algorithm with  $O(1/\epsilon)$ convergence for constrained composite convex programs. The new algorithm is faster than the classical primal-dual subgradient algorithm and the dual subgradient algorithm, both of which have an $O(1/\epsilon^2)$ convergence time.  The new algorithm has the same convergence time as that of a parallel algorithm recently proposed in \cite{YuNeely17SIOPT} for convex programs with separable objective and constraint functions. However, if the objective or constraint function is not separable, the algorithm in \cite{YuNeely17SIOPT} is no longer parallel and each iteration requires to solve a set constrained convex program. In contrast, the algorithm proposed in this paper is still parallel when the convex program is smooth or the non-smooth part is separable.  In these cases, the new algorithm has much smaller per-iteration complexity than the algorithm in \cite{YuNeely17SIOPT}.

\bibliographystyle{siamplain} 
\bibliography{mybibfile}

\end{document}